\documentclass{article}

\usepackage{amsthm,amsmath,amssymb}
\usepackage[numbers]{natbib}
\usepackage[colorlinks,citecolor=blue,urlcolor=blue]{hyperref}
\usepackage{mathabx}
\usepackage{stmaryrd}

\newtheorem{prop}{Propositon}[subsection]

\newtheorem{lem}{Lemma}[subsection]

\newtheorem*{theofinal}{Final theorem}
\newtheorem{theo}{Theorem}[subsection]

\newtheorem{cor}{Corollary}
\theoremstyle{remark}
\newtheorem{remark}{Remark}
\newtheorem{example}{Example}
\newtheorem{OpQu}{Open question}
\newtheorem{Suppl}{Supplement}

\newcommand{\mref}[2]{\ref{#1-#2}~(\ref{sec-#2})}

\newcommand{\zzl}{\lgroup}
\newcommand{\zzr}{\rgroup}
\newcommand{\col}[2]{\zzl#1\mid #2\zzr}
\newcommand{\scol}[2]{\left(#1\mid #2\right)}
\newcommand{\add}[2]{\bigcup\zzl #1 \mid #2\zzr}

\DeclareMathOperator{\infraDprod}{infra-\mathcal{D}-prod}
\DeclareMathOperator{\infraEprod}{infra-\mathcal{E}-prod}
\DeclareMathOperator{\infraEalprod}{infra-\mathcal{E}_{\alpha}-prod}
\DeclareMathOperator{\infraDpower}{infra-\mathcal{D}-power}
\DeclareMathOperator{\infraEalpower}{infra-\mathcal{E}_{\alpha}-power}
\DeclareMathOperator{\coz}{coz}
\DeclareMathOperator{\zer}{zer}

\newcommand{\al}{\alpha}
\newcommand{\gm}{\gamma}
\newcommand{\de}{\delta}
\newcommand{\eps}{\varepsilon}
\newcommand{\kap}{\varkappa}
\newcommand{\lm}{\lambda}
\newcommand{\si}{\sigma}
\newcommand{\vphi}{\varphi}
\newcommand{\om}{\omega}

\newcommand{\cD}{\mathcal{D}}
\newcommand{\cE}{\mathcal{E}}
\newcommand{\cP}{\mathcal{P}}

\newcommand{\Nbb}{\mathbb{N}}
\newcommand{\Zbb}{\mathbb{Z}}
\newcommand{\Qbb}{\mathbb{Q}}
\newcommand{\Rbb}{\mathbb{R}}

\newcommand{\hR}{\widehat{\mathbb{R}}}
\newcommand{\hS}{S}
\newcommand{\hRe}{R}

\newcommand{\impl}{\Rightarrow}
\newcommand{\iimpl}{\Leftrightarrow}
\newcommand{\razn}{\setminus}
\newcommand{\eq}{\equiv}
\newcommand{\vrn}{\varnothing}

\newcommand{\detau}{\mathrel{\delta_\tau}}
\newcommand{\depi}{\mathrel{\delta_\pi}}
\newcommand{\dee}{\mathrel{\delta}}
\newcommand{\dest}{\mathrel{\delta^{st}}}
\newcommand{\dekap}{\mathrel{\delta_\varkappa}}
\newcommand{\derho}{\mathrel{\delta_\rho}}
\newcommand{\delm}{\mathrel{\delta_\lambda}}
\newcommand{\epstau}{\mathrel{\varepsilon_\tau}}
\newcommand{\epst}{\mathrel{\varepsilon^{st}}}
\newcommand{\epskap}{\mathrel{\varepsilon_\varkappa}}
\newcommand{\epsrho}{\mathrel{\varepsilon_\rho}}
\newcommand{\epslm}{\mathrel{\varepsilon_\lambda}}
\newcommand{\inn}{\inplus}

\begin{document}

\begin{center}

\textbf{\large Infrafiltration theorem and some closed inductive sequence of models of generalized second-order Dedekind theory of real numbers with exponentially increasing powers}\\[1mm]

\textbf{Valeriy K. Zakharov}\footnote{valeriy\_zakharov@list.ru; Faculty of Mathematics and Mechanics, Lomonosov Moscow State University, Moscow, Russia},
\textbf{Timofey V. Rodionov}\footnote{t.v.rodionov@gmail.com; Faculty of Mathematics and Mechanics, Lomonosov Moscow State University, Moscow, Russia}

\end{center}

\begin{abstract}
The paper is devoted to construction of some closed inductive sequence of models of the generalized second-order Dedekind theory of real numbers with exponentially increasing powers. These models are not isomorphic whereas all models of the standard second-order Dedekind theory are. 
The main idea in passing to generalized models is to consider instead of superstructures with the single common set-theoretical equality and the single common set-theoretical belonging superstructures with several generalized equalities and several generalized belongings for first and second orders.
The basic tools for the presented construction are the infraproduct of collection of mathematical systems different from the factorized Lo\'s ultraproduct and the corresponding generalized infrafiltration theorem. As its auxiliary corollary we obtain the generalized compactness theorem for the generalized second-order language.\\


\end{abstract}

\section{Introduction}

It is well known that all \textit{standard} models of the standard second-order Dedekind theory of real numbers are isomorphic (see, for example, \cite[7.2]{Fef1963}). The paper is devoted to the exposition of some \textit{generalized} second-order Dedekind theory of real numbers with non-isomorphic generalized models.

More precisely, the paper is devoted to construction of some closed inductive sequence $\hRe_i$ ($1\le i\le\omega_0$) of models of the generalized second-order Dedekind theory of real numbers with exponentially increasing powers. 
The models $\hRe_i$ ($0\le i<\omega_0$) are embedded in each other as submodels and at the same time they all are embedded in the \emph{limit-closer model}~$\hRe_{\omega_0}$ as extending submodels.
These generalized models are completely different from mathematical systems presented in~\cite[2.14]{Mendelson1997} under the name of \emph{non-standard analysis}.

The main idea in passing to generalized models is to consider the \emph{generalized second-order language~$L(\Sigma_2^g)$} of some \emph{generalized signature~$\Sigma_2^g$} containing, in addition to individual and predicative constants and variables, some symbols~$\detau$ of \emph{generalized equalities} and some symbols~$\epstau$ of \emph{generalized belongings} for first-order types~$\tau$ and second-order types~$\tau\eq[\tau_0,\ldots,\tau_k]$.

Correspondingly, in the capacity of initial formulas of the language $L(\Sigma_2^g)$ the formulas of the following two forms are taken: the formula $y^\sigma\mathrel{\delta_\sigma} z^\sigma$ and the formula $(x_0^{\tau_0},\ldots,x_k^{\tau_k})\epstau u^\tau$, where~$y^\sigma$ and~$z^\sigma$ are the variables of the first- or the second-order type~$\sigma$ and~$x_i^{\tau_i}$ and $u^\tau$ are the variables of the first-order types~$\tau_i$ and the second-order type~$\tau\equiv[\tau_0,\ldots,\tau_k]$, respectively.

These atomic formulas are interpreted on an evaluated system $\zzl \zzl A,S_2^g\zzr,\gamma\zzr$ (with a superstructure~$S_2^g$ of the signature~$\Sigma_2^g$ over a support~$A$ and an evaluation~$\gamma$ on the system $U\equiv\zzl A,S_2^g\zzr$) in the following generalized way: $\gamma(y^{\si})\approx_\sigma\gamma(z^{\si})$ and $(\gamma(x_0^{\tau_0}),\ldots,\gamma(x_k^{\tau_k}))\inn_\tau\gamma(u^{\tau})$, where $\approx_\sigma$ is a \emph{generalized ratio of equality} and $\inn_\tau$ is a \emph{generalized ratio of belonging}. Generalized equalities and generalized belongings are connected with each other by the \emph{initial principle of change of equals} (see axiom E4 from~\ref{sec-tftsf-signform}).

With respect to the signature $\Sigma_2^g$ formulas $\varphi$ in the language $L(\Sigma_2^g)$ are defined by common induction, when we start from the above-mentioned atomic formulas.

To give a semantics of the language $L(\Sigma_2^g)$ a \emph{satisfaction of a formula $\varphi$  on the system~$U$ with respect to the evaluation of variables~$\gamma$} is defined according to the above-mentioned generalized interpretation of the atomic formulas.

The semantics for the language $L(\Sigma_2^g)$ differs both from the standard semantics (see \cite[Appendix]{Mendelson1997}, \cite[\S 16]{Takeuti2013}) and from the Henkin semantics (see \cite[Appendix]{Mendelson1997}, \cite[\S 21]{Takeuti2013}, \cite[\textbf{4}]{Dalen1997}, and \cite{Rossberg2004,Shapiro1991,Vaananen2001}), which restricts the range of values of the evaluation $\gamma(x^\tau)$ for a variable $x^\tau$ of a second-order type $\tau$ by some subset of the set $\mathcal{P}(\tau(A))$ of the \emph{terminal~$\tau(A)$}.

The general material about second-order notions mentioned above is presented in Sections~\ref{sec-tftsf} and~\ref{sec-systSigma} of the paper. More specific material about the \emph{generalized second-order Dedekind theory of real numbers~$Th^g_{R2}$} and about the canonical model 
$\hRe_0\eq R_2^g\eq\zzl\Rbb,S_{R2}\zzr$ is presented in Section~\ref{sec-Ded}. 

In Section \ref{sec-main} we construct some inductive sequence of non-canonical models $\hRe_i\eq\zzl\hR_i,\hS_i\zzr$, $1\le i\le\omega_0$, with exponentially increasing powers. The basic tool for construction of these systems is the \emph{infraproduct of collection of systems of the signature~$\Sigma_2^g$}, different from the factorized ultraproduct \`a la Lo\'s. To prove that the systems~$\hRe_i$ are models for~$Th^g_{R2}$ we use the simplified variant of the generalized infrafiltration theorem for the generalized second-order language~$L(\Sigma_2^g)$ presented in~\cite{Zakharov2008comp,ZakhYash2014}. Note that the corresponding proof of the infrafiltration property for the standard second-order language~$L(\Sigma_2^{st})$ do not ``pass''. 

\medskip

Further, to shorten the writings we use for the designation of a symbol-string~$\rho$ by a symbol-string~$\si$ the symbol-strings $\si\eq\rho$ or $\rho\eq\si$ ($\si$ is a \emph{designation for~$\rho$}).

\section{The type theory in the language of the signature with generalized equalities and belongings}\label{sec-tftsf}

\subsection{Types}\label{sec-tftsf-types}
\mbox{}

Fix the canonical set $\omega_0$ of all natural numbers and its subset $\Nbb\eq\omega_0\razn\{0\}$ constructed in the Neumann\,--\,Bernays\,--\,G\"odel (NBG) or Zermelo\,--\,Fraenkel (ZF) set theories or in the local theory of sets (LTS) (see~\cite{Zakharov2005LTS} and~\cite[1.1, A.2, B.1]{ZakhRodi2018SFM1}). Hereinafter ST denotes any of these set theories.

Define  by induction the \emph{semitypes} and the \emph{types}:
\begin{enumerate}
	\item
	$0$ is the \emph{semitype} and the \emph{type};
	\item
	if $\tau$ is a type, then $\tau$ is the \emph{semitype}:
	\item
	if $\tau$ is a semitype, then $[\tau]$ is the \emph{type};
	\item
	if $\tau_0,\ldots,\tau_k$ are semitypes and $k\ge1$, then $(\tau_0,\ldots,\tau_k)$ is the \emph{semitype}.
\end{enumerate}
This definition is a slight modification of the corresponding definition from~\cite[\S\,20]{Takeuti2013}.

Further, instead of $[(\tau_0,\ldots,\tau_k)]$ we shall write simply $[\tau_0,\ldots,\tau_k]$; then the notation $[\tau_0,\ldots,\tau_k]$ may be used for $k\ge0$.

Semantics of semitypes and types will be explained in the next subsection.

Types $0$ will be called the \emph{first-order type}. If $\tau_0,\ldots,\tau_k$ are first-order types and $k\ge 0$ then $[\tau_0,\ldots,\tau_k]$ will be called the \emph{second-order type}.

For a type $\tau\equiv[\tau_0,\ldots,\tau_k]$ with $k\ge0$ the types $\tau_0,\ldots,\tau_k$ will be called the \emph{parents of the type} $\tau$ and will be denoted by $p_0\tau,\ldots,p_k\tau$, respectively.
Consider the \emph{set $P(\tau)\equiv\{p_0\tau,\ldots,p_k\tau\}$ of all parents of the type~$\tau$}.

For the first-order type $\tau$ put formally $p\tau\equiv\tau$ and $P(\tau)\equiv\{p\tau\}=\{\tau\}$.

With any type  $\tau$ we associate \emph{the semitype $\check{\tau}$  of the type~$\tau$} as follows:
\begin{enumerate}
	\item
	if $\tau$ is the first-order type, then $\check{\tau}\equiv\tau$;
	\item
	if $\tau=[\tau_1]$ and $\tau_1$ is a semitype, then $\check{\tau}\equiv\tau_1$.
\end{enumerate}
In other words, the semitype of a type is obtained by omitting the square brackets.

\subsection{Terminals over set and mappings}\label{sec-tftsf-formterm}
\mbox{}

Define the \emph{terminals $\tau(A)$ of the semitypes $\tau$ over a set~$A$} by induction:
\begin{enumerate}
	\item
	$0(A)\equiv A$;
	\item
	if $\tau$ is a semitype, then $[\tau](A)\equiv\cP(\tau(A))$, where $\cP$ denotes the operation of taking power-set of the intended set;
	\item
	if $\tau_0,\ldots,\tau_k$ are semitypes, $k\ge 1$, then $(\tau_0,\ldots,\tau_k)(A)\equiv\tau_0(A)\times\ldots\times\tau_k(A)$.
\end{enumerate}

Thus, for semitypes $\tau_0,\ldots,\tau_k$ with $k\ge1$, for the type $\tau\equiv[\tau_0,\ldots,\tau_k]$, and for its semitype $\check{\tau}=(\tau_0,\ldots,\tau_k)$ the equalities $\tau(A)=\cP(\tau_0(A)\times\ldots\times\tau_k(A))$ and $\check{\tau}(A)=\tau_0(A)\times\ldots\times\tau_k(A)$ are fulfilled.  

Let $u:A\to B$ be a mapping from the set $A$ to the set $B$. Define the \emph{terminals $\tau^m(u)$ of the semitypes~$\tau$ over the mapping $u:A\to B$} by induction:
\begin{enumerate}
	\item
	$0^m(u)\equiv u:A\to B$;
	\item
	if $\tau$ is a semitype, then $[\tau]^m(u):\cP(\tau(A))\to\cP(\tau(B))$ is the mapping such that $[\tau]^m(u)(P)\eq(\tau^m(u))[P]\eq\{q\in\tau(B)\mid \exists\,p\in P\ (q=\tau^m(u)(p))\}$ for every $P\in\cP(\tau(A))$;
	\item
	if $\tau_0,\ldots,\tau_k$ are semitypes and $k\ge 1$, then
	 $$
	  (\tau_0,\ldots,\tau_k)^m(u):\tau_0(A)\times\ldots\times\tau_k(A)\to\tau_0(B)\times\ldots\times\tau_k(B)
	 $$
	   is the mapping such that
	$$
	 ((\tau_0,\ldots,\tau_k)^m(u))(p_0,\ldots,p_k)\eq(\tau_0^m(p_0),\ldots,\tau_k^m(p_k))
	$$
	for every $(p_0,\ldots,p_k)\in\tau_0(A)\times\ldots\times\tau_k(A)$.
\end{enumerate}

\subsection{The signature with generalized equalities and belongings and its language}\label{sec-tftsf-signform}
\mbox{}

A non-empty set $\Theta$ of types $\tau$ will be called the \emph{type domain} if $\tau\in\Theta$ implies $p\tau\in\Theta$ for every parent $p\tau$ of the type $\tau$. In the type domain $\Theta$ select the \emph{belonging type subdomain} 
$\Theta_b\equiv \{\tau\in\Theta\mid\exists\, k\in\omega_0\ \exists\tau_0,\ldots,\tau_k\in\Theta\ (\tau=[\tau_0,\ldots,\tau_k])\}$.

A collection $\Sigma_c\equiv\col{\Sigma_c^\tau}{\tau\in\Theta}$ of \emph{collections $\Sigma_c^\tau\equiv\col{\sigma_\omega^\tau}{\omega\in\Omega_\tau}$ of constants $\sigma_\omega^\tau$ of the types~$\tau$} will be called the \emph{signature of constants of the type domain~$\Theta$}. Sets $\Omega_\tau$ may be empty, and then $\Sigma_c^\tau=\vrn$.

The constants $\sigma_\omega^0$ of the first-order type $0$ are called \emph{individual} or \emph{objective}. The constants of other types are called \emph{predicate}.

A collection $\Sigma_e\equiv\col{\detau}{\tau\in\Theta}$ of binary \emph{predicate symbols of} (\emph{generalized}) \emph{equalities~$\detau$ of the types~$\tau$} will be called the \emph{signature of} (\emph{generalized}) \emph{equalities of the type domain~$\Theta$}. It follows from the definition of the type domain that for every equality symbol $\detau$ the collection $\Sigma_e$ contains necessarily the equality symbols  $\delta_{p\tau}$ for every parent~$p\tau$ of the type~$\tau$.

A collection $\Sigma_b\equiv\col{\epstau}{\tau\in\Theta_b}$ of binary \emph{predicate symbols of} (\emph{generalized}) \emph{belongings~$\epstau$ of the types~$\tau$} will be called the \emph{signature of} (\emph{generalized}) \emph{belongings of the type domain~$\Theta$}.

A collection $\Sigma_v\equiv\col{\Sigma_v^\tau}{\tau\in\Theta}$ of denumerable  \emph{sets $\Sigma_v^\tau$  of variables $x^\tau$, $y^\tau$,\ldots of the types~$\tau$} will be called the \emph{signature of variables of the type domain~$\Theta$}. The sets $\Sigma_v^\tau$ may be empty. The variables $x^0,y^0,\ldots$ of the first-order type~$0$  are called \emph{individual} or \emph{objective}. The variables of other types are called \emph{predicate}.

Further, we shall always assume that for every type $\tau\in\Theta$ there are either constants or variables of this type.

The quadruple $\Sigma^g\equiv\Sigma_c|\Sigma_e|\Sigma_b|\Sigma_v$ will be called the \emph{generalized signature} or the \emph{signature with generalized equalities and belongings}.

The \emph{language $L(\Sigma^g)$ of the generalized signature} $\Sigma^g$ consists of:
\begin{enumerate}
	\item
	all types $\tau$ from the type domain $\Theta$;
	\item
	all members of all signatures from $\Sigma^g$;
	\item
	the logical symbols $\lnot$, $\lor$, $\land$, $\Rightarrow$, $\forall$, and $\exists$;
	\item
	parenthesis.
\end{enumerate}

If the type domain $\Theta$ contains first- and second-order types only and at least one second-order type, then we shall say that the signature $\Sigma^g$ and the language $L(\Sigma^g)$ have the \emph{second order} (see~\cite[Appendix]{Mendelson1997}, \cite[\textbf{4}]{Dalen1997}). In this case the notations $\Sigma_2^g$ and $L(\Sigma_2^g)$ will be used.

\subsection{Terms, formulas, and the type theory for the language of the generalized signature}
\mbox{}

Constants and variables of a type $\tau$ are called \emph{terms of the type $\tau$ of the language $L(\Sigma^g)$}.

The \emph{atomic formulas of the language}  $L(\Sigma^g)$ are defined in the following way:
\begin{enumerate}
	\item
	if $q$ and $r$ are terms of a type $\tau\in\Theta$, then $q\detau r$ is an \emph{atomic formula};
	\item
	if $\tau_0$,\ldots, $\tau_k$ are types from $\Theta$ for $k\ge 0$, $\tau\equiv[\tau_0,\ldots,\tau_k]\in\Theta_b$, $q_0^{\tau_0}$, \ldots, $q_k^{\tau_k}$ are terms of the types $\tau_0$, \ldots, $\tau_k$, respectively, and $r^\tau$ is a term of the type~$\tau$, then $(q_0^{\tau_0},\ldots,q_k^{\tau_k})\epstau r^\tau$ is the \emph{atomic formula};
	in particular, for $k=0$ the symbol-string  $q_0^{\tau_0}\mathrel{\eps_{[\tau_0]}} r^{[\tau_0]}$ is the \emph{atomic formula}.
\end{enumerate}

The \emph{formulas of the language} $L(\Sigma^g)$ are constructed from atomic ones with the use of connectives $\lor$, $\land$, $\lnot$, $\Rightarrow$, quantifiers $\exists x^\tau$ and $\forall x^\tau$ with respect to the variables $x^\tau$, and parenthesis.

The \emph{logical axiom schemes of \emph{the} type theory in the language $L(\Sigma^g)$ of the generalized signature $\Sigma^g$} are the schemes of the predicate calculus, where variables and terms substituting each other must be of the same type $\tau\in\Theta$.

In addition to these axiom schemes, consider the following \emph{equality axioms for the types $\tau\in\Theta$}.

\textbf{E1.} $\forall\,x^\tau\ (x\detau x)$.

\textbf{E2.} $\forall\,x^\tau,y^\tau\ (x\detau y\Rightarrow y\detau x)$.

\textbf{E3.} $\forall\, x^\tau,y^\tau,z^\tau\ (x\detau y\land y\detau z\Rightarrow x\detau z)$.

\textbf{E4.} (The \emph{initial principle of change of equals}.) 
\begin{multline*}
\forall\, x_0^{\tau_0},y_0^{\tau_0},\ldots,x_k^{\tau_k},y_k^{\tau_k},u^\tau,v^\tau\ \bigl(x_0\mathrel{\delta_{\tau_0}}y_0\land\ldots\land x_k\mathrel{\delta_{\tau_k}}y_k\land u\detau v\Rightarrow\\
\Rightarrow((x_0,\ldots,x_k)\epstau u\Leftrightarrow(y_0,\ldots,y_k)\eps_\tau v))\bigr), \text{ where } \tau\equiv[\tau_0,\ldots,\tau_k].
\end{multline*}

The \emph{inference rules} in the depicted type theory are:

$$\frac{\varphi,\,\varphi\Rightarrow\psi}{\psi}\ (MP)\quad\mbox{ and }\quad\frac{\varphi(x^\tau)}{\forall\, x^\tau\ \varphi(x^\tau)}\quad(Gen).$$

If there are non-logical axioms or axiom schemes written by second-order formulas of the language $L(\Sigma^g_2)$, then we shall say that a (\emph{mathematical}) \emph{generalized second-order theory} is given.

\section{Mathematical systems of the signature $\Sigma^g$ with generalized equalities and belongings}\label{sec-systSigma}

\subsection{The definition of mathematical systems and their homomorphisms of the generalized signature $\Sigma^g$}\label{sec-systSigma-defsystSigma}
\mbox{}

\paragraph*{Generalized systems.}
Let $\Sigma^g$ be a fixed signature defined in~\ref{sec-tftsf-signform}. Fix also a set~$A$.
For the set~$A$ and the signature $\Sigma^g$ consider the following collections:
\begin{enumerate}
	\item
	 $S_c\equiv\col{S_c^\tau}{\tau\in\Theta}$ of collections $S_c^\tau\equiv\col{s_\omega^\tau}{\omega\in\Omega_\tau}$ of \emph{constant structures} $s_\omega^\tau\in\tau(A)$ \emph{of the types}~$\tau$;
	\item
	 $S_e\equiv\col{\approx_\tau}{\tau\in\Theta}$ of \emph{generalized ratios of equality} $\approx_\tau\subset\tau(A)\times\tau(A)$ \emph{of the types~$\tau$ on the sets}~$\tau(A)$, containing the usual set-theoretic ratios of equality~$=$ on the sets~$\tau(X)$, i.\,e., such ratios $\approx_\tau$ that for every elements $r,s\in\tau(A)$ the equality $r=s$ implies the generalized equality $r\approx_\tau s$;
	\item
	 $S_b\equiv\col{\inn_\tau}{\tau\in\Theta_b}$ of \emph{generalized ratios of belonging $\inn_\tau\subset\check{\tau}(A)\times\tau(A)$ of the types}~$\tau$, containing the usual set-theoretic ratios of belonging~$\in$ from the sets~$\check{\tau}(X)$ into the sets~$\tau(X)$, i.\,e., such ratios~$\inn_\tau$ that for every elements $p\in\check{\tau}(A)$ and $P\in\tau(A)$ the belonging $p\in P$ implies the generalized belonging $p\inn_\tau P$;
	\item
	 $S_v\equiv\col{\tau(A)}{\tau\in\Theta}$ of the terminals $\tau(A)$ of the types~$\tau$ over the set~$A$.
\end{enumerate}

The quadruple $S\equiv\zzl S_c,S_e,S_b,S_v\zzr$ of the above-mentioned collections will be called a \emph{superstructure of the signature~$\Sigma^g$ over the set~$A$}.

The pair $U\equiv\zzl A,S\zzr$ will be called a \emph{mathematical system of the generalized signature~$\Sigma^g$ with the support \textup{(\emph{carrier})}~$A$ and the superstructure~$S$.} This notion is a generalization of the notion of an \emph{algebraic system of the signature}~$\Sigma_1$ (see~\cite[\S\,15]{ErshPal1984}).

The mathematical system $U\equiv\zzl A,S\zzr$ will be called also an \emph{interpretation of the signature~$\Sigma^g$ on the support~$A$}.

Further, for a type $\tau=[\tau_0,\ldots,\tau_k]$ and elements $p\equiv(p(0),\ldots,p(k))$, $q\equiv(q(0),\ldots,q(k))\in\check{\tau}(A)=\tau_0(A)\times\ldots\times\tau_k(A)$ along with 
$$
 p(0)\approx_{\tau_0}q(0)\land\ldots\land p(k)\approx_{\tau_k}q(k)
$$
 we shall also write $p\approx_{\check{\tau}}q$.

The generalized equalities $\approx_\tau$ and the generalized belongings $\inn_\tau$ admit some additional conditions.

A system $U$ will be called \emph{balanced} if 
$$
 \forall\,P,Q\in\tau(A)\ (P\approx_\tau Q\Leftrightarrow\forall\, p\in P\ \exists\, q\in Q\ (q\approx_{\check{\tau}}p)\land\forall\, q\in Q\ \exists\, p\in P\ (p\approx_{\check{\tau}}q)),
$$
 where $\tau_0,\ldots,\tau_k\in\Theta$, $k\ge 0$ and $\tau\equiv[\tau_0,\ldots,\tau_k]\in\Theta$.
A system $U$ will be called \emph{regular} if
$
 \forall\, p\in\check{\tau}(A)\ \forall\, P\in\tau(A)\ (p\inn_\tau P\Leftrightarrow \exists\, q\in P\ (p\approx_{\check{\tau}} q)),
$
  where $\tau_0,\ldots,\tau_k\in\Theta$, $k\ge 0$, and $\tau\equiv[\tau_0,\ldots,\tau_k]\in\Theta$.
A system $U$ will be called \emph{extensional} if
$$
 \forall\, P,Q\in\tau(A)\ (P\approx_\tau Q\Leftrightarrow\forall\, p\ (p\inn_\tau P\Rightarrow p\inn_\tau Q)\land\forall\, q\ (q\inn_\tau Q\Rightarrow q\inn_\tau P)),
$$
where $\tau\in\Theta_b$.

\paragraph*{Generalized homomorphisms.}
Let $U\eq\zzl A,S\zzr$ and $V\eq\zzl B,T\zzr$ be systems of the signature~$\Sigma^g$ from~\ref{sec-systSigma-defsystSigma}. A mapping $u:A\to B$ in the considered set theory ST from the set~$A$ to the set~$B$ is called a \emph{homomorphism of the signature~$\Sigma^g$ from the system~$U$ into the system~$V$} if for every type $\tau\in\Theta$, every index $\omega\in\Omega_{\tau}$, every corresponding constant structure $s_{\omega}^{\tau}\in\tau(A)$ of the collection~$S_c$, and every corresponding constant structure $t_{\omega}^{\tau}\in\tau(B)$ of the collection~$T_c$ the following properties are fulfilled:
\begin{enumerate}
	\item if $\tau=0$, then $\tau^m(u)(s_{\omega}^{\tau})=u(s_{\omega}^{\tau})=t_{\omega}^{\tau}$;
	\item if $\tau\in\Theta_b$, then every generalized belonging $p\inn_{\tau,A} s_{\omega}^{\tau}$ implies the corresponding generalized belonging $\check{\tau}^m(u)(p)\inn_{\tau,B} t_{\omega}^{\tau}$ for every $p\in\check{\tau}(A)$.
\end{enumerate}

\subsection{Evaluations and models}\label{sec-systSigma-eval}
\mbox{}

An \emph{evaluation on a system $U\equiv\zzl A,S\zzr$ of the signature}~$\Sigma^g$ is a mapping~$\gamma$ defined on the set of all variables of the signature~$\Sigma^g$ and associating with the variable $x^\tau$ of the type $\tau\in\Theta$ the element
$\gamma(x^\tau)$ of the terminal~$\tau(X)$ (see \cite[\S\,16]{ErshPal1984}, \cite[16.17]{Takeuti2013}).
The pair $\zzl U,\gamma\zzr$ consisting of the system~$U$ of the signature~$\Sigma^g$ and the evaluation~$\gamma$ on~$U$ will be called an \emph{evaluated mathematical system of the signature}~$\Sigma^g$.

Define the \emph{value $q[\gamma]$ of a term $q$ with respect to the evaluation~$\gamma$ on the system}~$U$ in the following way (see \cite[\S\,16]{ErshPal1984}, \cite[\S\,6]{Maltsev1973}, \cite[2.2]{Mendelson1997}, \cite[2.5]{Shoenfield2001}):
for a constant $\sigma^\tau_{\omega}$  of a type $\tau\in\Theta$ put  $\sigma^\tau_{\omega}[\gamma]\equiv s^\tau_{\omega}$ and for a variable  $x^\tau$ of a type $\tau\in\Theta$ put $x^\tau[\gamma]\equiv\gamma(x^\tau)$.

Define the \emph{satisfaction \emph{(\emph{translation})} of a formula $\varphi$ of the language~$L(\Sigma_2^g)$ on a system~$U$ of the signature~$\Sigma_2^g$ with respect to an evaluation~$\gamma$} (in notation, $U\vDash\varphi[\gamma]$) by induction in the following way (see \cite[2.2]{Mendelson1997}, \cite[2.5]{Shoenfield2001}, \cite[16.17]{Takeuti2013}, \cite[A.1.3]{ZakhRodi2018SFM1}):
\begin{enumerate}
	\item
	if $q$ and $r$ are terms of a type $\tau\in\Theta$ and $\varphi\equiv (q\detau r)$, then $U\vDash \varphi[\gamma]$ is equivalent to $q[\gamma]\approx_\tau r[\gamma]$;
	\item
	if $\tau_0,\ldots,\tau_k$ are types from $\Theta$ for $k\ge0$, $\tau\equiv[\tau_0,\ldots,\tau_k]\in\Theta$, $q_0,\ldots,q_k$ are terms of the types $\tau_0$, \ldots,$\tau_k$, respectively, $r$ is a term of the type $\tau$, and $\varphi\equiv(q_0,\ldots,q_k)\epstau r$, then $U\vDash\varphi[\gamma]$ iff $(q_0[\gamma],\ldots,q_k[\gamma])\inn_\tau r[\gamma]$;
	\item
	if $\varphi\equiv\lnot\psi$, then $U\vDash\varphi[\gamma]$ iff $U\vDash\psi[\gamma]$ is not true;
	\item
	if $\varphi\equiv(\psi\lor\xi)$, then $U\vDash\varphi[\gamma]$ iff $U\vDash\psi[\gamma]$ or $U\vDash\xi[\gamma]$;
	\item
	if $\varphi\equiv(\psi\land\xi)$, then $U\vDash\varphi[\gamma]$ iff $U\vDash\psi[\gamma]$ and $U\vDash\xi[\gamma]$;
	\item
	if $\varphi\equiv(\psi\Rightarrow\xi)$, then $U\vDash\varphi[\gamma]$ iff that $U\vDash\psi[\gamma]$ implies $U\vDash\xi[\gamma]$;
	\item
	if $\varphi\equiv\exists\, x^\tau\psi$, then $U\vDash\varphi[\gamma]$ is equivalent to $U\vDash\psi[\gamma']$ for some  evaluation $\gamma'$ such that $\gamma'(y^\sigma)=\gamma(y^\sigma)$ for every variable $y^\sigma\ne x^\tau$;
	\item
	if $\varphi\equiv\forall\, x^\tau\psi$, then $U\vDash\varphi[\gamma]$ is equivalent to $U\vDash\psi[\gamma']$ for every  evaluation $\gamma'$ such that $\gamma'(y^\sigma)=\gamma(y^\sigma)$ for every variable $y^\sigma\ne x^\tau$.
\end{enumerate}

Let $\Phi$ be a set of formulas of the language $L(\Sigma_2^g)$. An evaluated mathematical system $\zzl U,\gamma\zzr$ of the signature~$\Sigma_2^g$ will be called an (\emph{evaluated}) \emph{model for the set}~$\Phi$ if $U\vDash\varphi[\gamma]$ for every formula $\varphi\in\Phi$ (see \cite[\S\,17]{ErshPal1984}). A mathematical system~$U$ of the signature~$\Sigma_2^g$ will be called a \emph{model for the set}~$\Phi$ if an evaluated mathematical system $\zzl U,\gamma\zzr$ is a model for the set~$\Phi$ for every evaluation~$\gamma$ on~$U$.

A model $\zzl U,\gamma\zzr$ will be called \emph{balanced}, \emph{regular}, \emph{extensional}, etc. if the system~$U$ is the same.

A model $\zzl U,\gamma\zzr$ for a set $\Phi$ will be called \emph{second-order} if at least one formula from~$\Phi$ contains at least one second-order variable.

Remark that if a system $U\equiv\zzl A,S\zzr$ is considered in an axiomatic set theory, then the satisfaction of a closed formula~$\varphi$ of the language $L(\Sigma_2^g)$ with respect to any evaluation~$\gamma$ is reduced to correctness of the relativization~$\varphi^r$ of $\varphi$ on the corresponding terminals of the support~$A$ in this set theory.
Here the correctness of~$\vphi^r$ means that~$\vphi^r$ is a deducible formula in this axiomatic set theory.

Thus, if $\Phi$ consists of closed formulas only, then~$U$ is a model for~$\Phi$ iff $\zzl U,\gamma\zzr$ is a model for~$\Phi$ for some (and, consequently, for any) evaluation~$\gm$.

In particular, since equality axioms E1--E4 are closed formulas, their relativizations E1$^r$--E4$^r$ take the following forms:
\begin{align*}
E1^r &\eq\ \forall\, x\in\tau(A)\ (x\approx_\tau x);\\
E2^r &\eq\ \forall\, x,y\in\tau(A)\ (x\approx_\tau y\Rightarrow y\approx_\tau x);\\
E3^r &\eq\ \forall\, x,y,z\in\tau(A)\ (x\approx_\tau y\land y\approx_\tau z\Rightarrow x\approx_\tau z);\\
E4^r &\eq\ \forall\, x_0,y_0\in\tau_0(A)\ldots\forall\, x_k,y_k\in\tau_k(A)\ \forall u,v\in\tau(A)\ (x_0\approx_{\tau_0}y_0\land\ldots\land\\
     &\mbox{}\qquad\land x_k\approx_{\tau_k}y_k\land u\approx_\tau v \Rightarrow((x_0,\ldots,x_k)\inn_\tau u\Leftrightarrow (y_0,\ldots,y_k)\inn_\tau v)),\\
&\text{ where } \tau\equiv[\tau_0,\ldots,\tau_k],\ k\ge 0,
\text{ and all types are in } \Theta.
\end{align*}

The satisfaction of formulas E1$^r$--E3$^r$ means that all generalized equalities $\approx_\tau$ are equivalence relations on corresponding sets~$\tau(A)$, and the satisfaction of formula~E4$^r$ means the initial principle of change of equals in the atomic formula with the generalized belonging $\inn_\tau$.

Further on, we shall say that a system $U$ of the signature~$\Sigma_2^g$ has \emph{true generalized equalities and belongings} if axioms E1--E4 from~\ref{sec-tftsf-signform} are satisfied on~$U$ with respect to some (and, consequently, to any) evaluation~$\gamma$. This means that formulas E1$^r$--E4$^r$ are correct for the system~$U$ in the used set theory.

\subsection{The generalized equality of values of evaluations and satisfiability}\label{sec-systSigma-satisf}
\mbox{}

For every formula $\varphi$ of the language $L(\Sigma^g_2)$ we define the formula $\varphi^*$ by induction:
\begin{enumerate}
	\item 
	$\varphi^*\equiv\varphi$ for every atomic formula $\varphi$;
	\item
	$(\psi\land\xi)^*\equiv \psi^*\land\xi^*$;
	\item
	$(\lnot\psi)^*\equiv\lnot\psi^*$;
	\item
	$(\exists x^\tau\psi)^*\equiv\exists x^\tau\psi^*$;
	\item
	$(\psi\lor\xi)^*\equiv\lnot(\lnot\psi^*\land\lnot\xi^*)$;
	\item
	$(\psi\Rightarrow\xi)^*\equiv\lnot(\psi^*\land\lnot\xi^*)$;
	\item
	$(\forall x^\tau\psi)^*\equiv \lnot(\exists x^\tau(\lnot\psi^*))$.
\end{enumerate}

A formula $\varphi$ is said to be \emph{normalizable} if for every mathematical $\Sigma_2^g$-system~$U$ and every  evaluation~$\gamma$ on~$U$ the following condition holds: $U\vDash\varphi[\gamma]\Leftrightarrow U\vDash\varphi^*[\gamma]$.

\begin{lem}\label{lem-1-systSigma-satisf}
	Let formulas $\psi$ and $\xi$ be normalizable. Then formulas $\psi\land\xi$,  $\lnot\psi$, $\psi\lor\xi$, $\psi\Rightarrow\xi$, $\forall x^\tau\psi$, and $\exists x^\tau\psi$ are normalizable as well.
\end{lem}

The proof of this lemma uses the definition of satisfiability and some well known tautologies only, so it is omitted.

\begin{prop}\label{prop-1-systSigma-satisf}
	Every formula of the language $L(\Sigma_2^g)$ of the generalized second-order signature $\Sigma_2^g$ is normalizable.
\end{prop}

\begin{proof}
	Denote by $\Phi$ the set of all formulas of the language $L(\Sigma_2^g)$. The subset of the set $\Phi$ consisting of formulas containing at most $n\in\omega_0$ logical symbols $\lnot$, $\land$, $\Rightarrow$, $\lor$, $\exists$, $\forall$, denote by $\Phi_n$. It is clear that $\Phi=\add{\Phi_n}{n\in\omega_0}$.
	
	Prove by the complete induction principle the following assertion $A(n)$: \emph{every formula $\varphi\in\Phi$ is normalizable.}
	
	If $n=0$, then the formula $\varphi$ is atomic, and so by the definition of the operation $\varphi\mapsto\varphi^*$ we have $\varphi^*\equiv \varphi$. Consequently, the assertion $A(0)$ is true.
	
	Suppose that for all $m<n$ the assertion $A(m)$ is true. Let $\varphi\in \Phi_n$. If $\varphi\equiv\psi\land\xi$, $\varphi\equiv\lnot\psi$, $\varphi\equiv\exists x^\tau\psi$, $\varphi\equiv\psi\lor\xi$, $\varphi\equiv\psi\Rightarrow\xi$, or $\varphi\equiv\forall x^\tau\psi$, then $\psi,\xi\in\Phi_{n-1}$. Therefore by the induction hypothesis, the formulas~$\psi$ and~$\xi$ are normalizable. By Lemma~\ref{lem-1-systSigma-satisf} the formula $\varphi$ is normalizable. Hence the assertion $A(n)$ is true.
\end{proof}

\begin{prop}\label{prop-2-systSigma-satisf}
	Let $U$ be a mathematical system of the second-order signature $\Sigma_2^g$ with true generalized equalities and belongings. 
	Then for every formula $\varphi$ of the language $L(\Sigma_2^g)$ and every evaluations~$\gamma$ and~$\delta$ on the system~$U$ such that $\gamma(x^\tau)\approx_\tau\delta(x^\tau)$ for every variable $x^\tau$ of every type $\tau\in\Theta$ the properties $U\vDash\varphi[\gamma]$ and $U\vDash\varphi[\delta]$ are equivalent.
\end{prop}

\begin{proof}
	The set of all formulas $\varphi$ of the language $L(\Sigma_2^g)$ constructed by induction from the atomic formulas with the use of connectives $\lnot$ and $\land$ and  quantifier $\exists$ denote by $\Psi$. The subset of the set $\Psi$ consisting of formulas containing at most $n\in\omega_0$ logical symbols $\lnot$, $\land$, and $\exists$ denote by $\Psi_n$. It is clear that $\Psi=\add{\Psi_n}{n\in\omega_0}$.
	
	Prove by the complete induction principle the assertion $A(n)$: \emph{for every formula $\varphi\in\Psi_n$ and every mentioned evaluations $\gamma$ and $\delta$ the assertion of the Proposition holds}.
	
	Let $n=0$ and $\varphi\in\Psi_0$. Then $\varphi$ is an atomic formula.
	At first consider the atomic formula $\varphi$ of the form $q^\tau\delta_\tau r^\tau$. Suppose that $q^\tau=x^\tau$ and $r^\tau=\sigma_\omega^\tau$. Then $U\vDash \varphi[\gamma]$ is equivalent to $\gamma(x)\approx_\tau s_\omega^\tau$ and $U\vDash\varphi[\delta]$ is equivalent to $\delta(x)\approx_\tau s_\omega^\tau$.
	Since, by our condition, $\gamma(x)\approx_\tau\delta(x)$, then assuming $U\vDash\varphi[\gamma]$ and using axioms~E2$^r$ and~E3$^r$ we infer $U\vDash\varphi[\delta]$. The inverse inference is checked in the same way. For the terms $q^\tau$ and $r^\tau$ of other forms the reasons are quite similar.
	
	Now, consider the atomic formula $\varphi$ of the form $(q_0^{\tau_0},\ldots,q_k^{\tau_k})\varepsilon_\tau r^k$ for the type $\tau\equiv[\tau_0,\ldots,\tau_k]\in\Theta_b$. Assume that $q_\lambda^{\tau_\lambda}=x_\lambda^{\tau_\lambda}$ and $r^\tau=u^\tau$ for some variables $x_\lambda$ and $u$. Then $U\vDash\varphi[\gamma]$ is equivalent to $(\gamma(x_0),\ldots,\gamma(x_k))\inn_\tau\gamma(u)$ and $U\vDash\varphi[\delta]$ is equivalent to $(\delta(x_0),\ldots,\delta(x_k))\inn_\tau\delta(u)$.
	
	Suppose $U\vDash\varphi[\gamma]$. Since, by our condition, $\gamma(x_\lambda^{\tau_\lambda})\approx_{\tau_\lambda}\delta(x_\lambda^{\tau_\lambda})$, then using axiom~E4$^r$, we infer $U\vDash\varphi[\delta]$. The inverse inference is checked in the same way. For the terms $q_\lambda^{\tau_\lambda}$ and $r^\tau$ of other kinds the reasons are quite similar.
	
	Assume that assertion $A(m)$ is true for every $m<n$. Let $\varphi\equiv\exists x^\tau\psi$. Then $\psi\in\Psi_{n-1}$. Let be given some evaluations $\gamma$ and $\delta$ such that $\gamma(x^\tau)\approx_\tau\delta(x^\tau)$.
	
	Suppose $U\vDash\varphi[\gamma]$. It is equivalent to $U\vDash\psi[\gamma']$ for some evaluation $\gamma'$ such that $\gamma'(y)=\gamma(y)$ for any $y^\sigma\ne x^\tau$.
	
	Define an evaluation $\delta'$ on $U$ setting $\delta'(y)\equiv\delta(y)$ for every $y^\sigma\ne x^\tau$ and $\delta'(x)\equiv\gamma'(x)$. Then $\delta'(y)=\delta(y)\approx_\sigma\gamma(y)=\gamma'(y)$ and $\delta'(x)=\gamma'(x)$, i.\,e., $\delta'(x)\approx_\tau\gamma'(x)$.
	
	Since $\delta'\approx\gamma'$ in the above indicated sense, by our condition, we conclude that
	$U\vDash\psi[\gamma']\Leftrightarrow U\vDash\psi[\delta']$. Consequently, we obtain the property $U\vDash\psi[\delta']$. By construction, $\delta'(y)=\delta(y)$ for every $y^\sigma\ne x^\tau$.
	By the definition of satisfiability, we conclude that $U\vDash\varphi[\delta]$. The inverse inference of $U\vDash\varphi[\gamma]$ from $U\vDash\varphi[\delta]$ is established quite analogously.
	
	Now, let $\varphi\equiv\psi\land\xi$. Then $\psi,\xi\in\Psi_{n-1}$, whence $U\vDash\psi[\gamma]\Leftrightarrow U\vDash\psi[\delta]$ and  $U\vDash\xi[\gamma]\Leftrightarrow U\vDash\xi[\delta]$. Hence $(U\vDash\psi[\gamma]\wedge U\vDash\xi[\gamma])\Leftrightarrow(U\vDash\psi[\delta]\wedge U\vDash\xi[\delta])$. Thus, $U\vDash\varphi[\gamma]\Leftrightarrow U\vDash\varphi[\delta]$.
	
	Finally, let $\varphi\equiv\lnot\psi$. Then $\psi\in\Psi_{n-1}$. Consequently, 
	$U\vDash\psi[\gamma]\Leftrightarrow U\vDash\psi[\delta]$. From here $U\vDash\varphi[\gamma]\Leftrightarrow\lnot(U\vDash\psi[\gamma])\Leftrightarrow\lnot(U\vDash\psi[\delta])\Leftrightarrow U\vDash\varphi[\delta]$.
	
	This proves that the assertion $A(n)$ is true. By the complete induction principle, the assertion $A(n)$ is true for every natural number $n\in\omega_0$, i.\,e., the assertion of the Proposition holds for every formula $\varphi\in\Psi$.
	
	Now let $\varphi$ be an arbitrary formula of the language  $L(\Sigma_2^g)$. By virtue of  Proposition~\ref{prop-1-systSigma-satisf} we have $U\vDash\varphi[\gamma]\Leftrightarrow U\vDash\varphi^*[\gamma]$ and $U\vDash\varphi[\delta]\Leftrightarrow U\vDash\varphi^*[\delta]$. By the definition of the operation $\varphi\mapsto\varphi^*$, we have $\varphi^*\in\Psi$. As was shown above, $U\vDash\varphi^*[\gamma]\Leftrightarrow U\vDash\varphi^*[\delta]$. As a result, we obtain the equivalence $U\vDash\varphi[\gamma]\Leftrightarrow U\vDash\varphi[\delta]$.
\end{proof}

\subsection{Examples of good models for the second-order equality axioms}\label{sec-systSigma-good}
\mbox{}

Construct for axioms E1--E4 two regular, balanced, extensional, second-order models.

Take $\rho\equiv 0$, $\sigma\equiv[\rho]$, $\Theta\equiv\{\rho,\sigma\}$, $\Omega_\rho=\varnothing$, $\Omega_\sigma=\varnothing$, $\Sigma_c^\rho=\varnothing$, and $\Sigma_c^\sigma=\varnothing$. Then $\Sigma_e\equiv(\delta_\rho,\delta_\sigma)$, $\Theta_b=\{\sigma\}$, $\Sigma_b\equiv\scol{\varepsilon_\tau}{\tau\in\Theta_b}$, i.\,e., $\Sigma_b$ consists of the symbol $\varepsilon_\sigma=\varepsilon_{[\rho]}$ only, and the collection $\Sigma_v\equiv\scol{\Sigma_v^\tau}{\tau\in\Theta}$ consists of a denumerable set $\Sigma_v^\rho$ of variables 
$x^\rho, y^\rho,\ldots$ of the first-order type~$\rho$ and a denumerable set $\Sigma_v^\sigma$ of variables $u^\sigma,v^\sigma,\ldots$ of the second-order type~$\sigma$.

Consider the signature $\Sigma\equiv\Sigma_c\mid\Sigma_e\mid\Sigma_b\mid\Sigma_v$. This language contains the three atomic formulas: $x^\rho\delta_\rho y^\rho$, $u^\sigma\delta_\sigma v^\sigma$ and $x^\rho\varepsilon_\sigma u^\sigma$.

\begin{example}
	Take the set $\Qbb\eq\Zbb\times(\Zbb\razn\{0\})$ of all rational fractions $p\eq\frac{m}{s}$ as the set~$A_1$. Since $\Omega_\rho=\Omega_\sigma=\varnothing$, there are no constants.
	
	For fractions  $p\eq\frac{m}{s}$ and $p\eq\frac{n}{t}$ put $p\approx_\rho q$ if $mt=ns$ in~$\Zbb$. For sets $P,Q\in\cP(A_1)$ put $P\approx_\sigma Q$ if 
	$(\forall\,p\in P\ \exists\,q\in Q\ (p\approx_\rho q))\land (\forall\,q\in Q\ \exists\,p\in P\ (q\approx_\rho p))$. 
	It is clear that the generalized ratio of equality~$\approx_{\si}$ is wider than the usual set-theoretical ratio of equality~$=$ in ST. For example, for $P_0\eq\{\frac{3}{8},\frac{2}{3}\}$ and $Q_0\eq\{\frac{6}{16},\frac{2}{3},\frac{4}{6}\}$ we have $P_0\approx_{\si}Q_0$ but $P_0\ne Q_0$.
	
	For a fraction $p\in A_1$ and a set $P\in\cP(A_1)$ put $p\inn_\sigma P$ if $\exists\,q\in A_1\ (q\approx_\rho p\land q\in P)$.
	It is clear that the generalized ratio of belonging~$\inn_{\si}$ is wider than the usual set-theoretical ratio of belonging~$\in$ in ST. For example, $\frac{6}{16}\inn_{\si}P_0$ and $\frac{6}{9}\inn_{\si}P_0$ but $\frac{6}{16}\notin P_0$ and
	$\frac{6}{9}\notin P_0$.
	
	The collection of terminals $S_v^1\equiv\col{\tau(A_1)}{\tau\in\Theta}$ consists of the terminal $\rho(A_1)=A_1$ and the terminal $\sigma(A_1)=\cP(A_1)$.
	
	The constructed collections form the superstructure $S_1$ over the set~$A_1$.
	Consider the mathematical system $U_1\equiv\zzl A_1,S_1\zzr$ of the signature~$\Sigma$.
\end{example}

\begin{example}
	Take the set of all closed segments $p$ of straight lines on the plane as the set~$A_2$. Since $\Omega_\rho=\Omega_\sigma=\varnothing$, there are no constants.
	
	For segments $p,q\in A_2$ put $p\approx_\rho q$ if $q$ is obtained from $p$ by some parallel transfer. For sets $P,Q\in\cP(A_2)$ of segments put $P\approx_\sigma Q$ if 
$$
 (\forall\,p\in P\ \exists\,q\in Q\ (p\approx_\rho q))\land (\forall\,q\in Q\ \exists\,p\in P\ (q\approx_\rho p)). 
$$
	
	For a segment $p\in A_2$ and a set of segments $P\in\cP(A_2)$ put $p\inn_\sigma P$ if 
	$\exists\,q\in A_2\ (q\approx_\rho p\land q\in P)$, i.\,e., the segment~$p$ can be transferred into the set~$P$ by some parallel transfer.
	
	The collection of terminals $S_v^2\equiv\col{\tau(A_2)}{\tau\in\Theta}$ consists of the terminal $\rho(A_2)=A_2$ and the terminal $\sigma(A_2)=\cP(A_2)$.
	
	The constructed collections form the superstructure $S_2$ over the set~$A_2$.
	Consider the mathematical system $U_2\equiv\zzl A_2,S_2\zzr$ of the signature~$\Sigma$.
\end{example}

\begin{prop}\label{prop-3-systSigma-good}
	The above-constructed mathematical systems $U_1$ and $U_2$ are the regular, balanced, extensional, second-order models for equality axioms \emph{E1--E4}.
\end{prop}

\begin{proof}
	The correctness of the equality axioms is evident. The regularity follows from the definition. The same is true for the balance property.
	
	Check the extensionality property. Let $P,Q \in\sigma(A)=\cP(A)$. Assume $p\in P$. Then $p\inn_\sigma P$. Suppose the right side of the extensionality formula. By condition we conclude $p\inn_\sigma Q$. By the regularity property there exists an element
	$q\in Q$ such that $q\approx_\rho p$. The inverse finding of an element $p\in P$ for a given element $q\in Q$ such that $p\approx_\rho q$ is established quite similarly. In accordance with the definition of the equality $\approx_\sigma$ we conclude that $P\approx_\sigma Q$. Thus, we have inferred the left side of the extensionality formula. It follows from the correctness of axiom~E4$^r$ that the left side implies the right one.
\end{proof}

\section{The generalized second-order Dedekind theory of real numbers}\label{sec-Ded}

\subsection{The signature for the generalized and the standard second-order Dedekind theories of real numbers}\label{sec-Ded-sign}
\mbox{}

Consider the first-order type $\pi\eq 0$, the second-order types $\kap\eq[\pi]$, $\rho\eq[\pi,\pi]$, and $\lm\eq[\pi,\pi,\pi]$ and the type domain $\Theta\eq\Theta^g_{R2}\eq\{\pi,\kap,\rho,\lm\}$ with the belonging type subdomain $\Theta_{b}\eq\{\kap,\rho,\lm\}$.

Put $\Omega_{\pi}\eq 2$, $\Omega_{\kap}\eq\vrn$, $\Omega_{\rho}\eq 3$, $\Omega_{\lm}\eq 2$, and consider the collections
\begin{align*}
\Sigma_c^{\pi}&\eq\scol{\si_{\omega}^{\pi}}{\omega\in\Omega_{\pi}}=(\si_0^{\pi},\si_1^{\pi}),& 
\Sigma_c^{\kap}&\eq\scol{\si_{\omega}^{\kap}}{\omega\in\Omega_{\kap}}=\vrn,\\ 
\Sigma_c^{\rho}&\eq\scol{\si_{\omega}^{\rho}}{\omega\in\Omega_{\rho}}=(\si_0^{\rho},\si_1^{\rho},\si_2^{\rho}),&\text{ and }
\Sigma_c^{\lm}&\eq\scol{\si_{\omega}^{\lm}}{\omega\in\Omega_{\lm}}=(\si_0^{\lm},\si_1^{\lm}).
\end{align*}
They compose the signature of constants of the type domain~$\Theta$ of the form
$\Sigma_c=\scol{\Sigma_c^{\tau}}{\tau\in\Theta}=
((\si_0^{\pi},\si_1^{\pi}),\vrn,(\si_0^{\rho},\si_1^{\rho},\si_2^{\rho}),(\si_0^{\lm},\si_1^{\lm}))$
containing the objective first-order constants $\si_0^{\pi}$ and $\si_1^{\pi}$ for denoting the \emph{real numbers~$0$} (\emph{null}) and~$1$ (\emph{unit}), respectively, the predicate second-order constants~$\si_0^{\rho}$, $\si_1^{\rho}$, and~$\si_2^{\rho}$ for denoting the \emph{ratio of negation}, the \emph{ratio of inversion}, and the \emph{ratio of order}, respectively, and the predicate second-order constants~$\si_0^{\lm}$ and~$\si_1^{\lm}$ for denoting the \emph{ratio of addition} and the \emph{ratio of multiplication}, respectively.

Further, along with $\si_0^{\pi}$, $\si_1^{\pi}$, $\si_0^{\rho}$, $\si_1^{\rho}$, $\si_2^{\rho}$, $\si_0^{\lm}$, and $\si_1^{\lm}$ we shall simply write~$0$, $1$, $-$, $/$, $\le$, $+$, and~$\cdot$, respectively.

Take the signature of the generalized equalities of the type domain~$\Theta$ of the form $\Sigma_e\eq\scol{\de_{\tau}}{\tau\in\Theta}=(\de_{\pi},\de_{\kap},\de_{\rho},\de_{\lm})$ containing the first-order equality~$\de_{\pi}$, and the second-order equalities $\de_{[\pi]}$, $\de_{[\pi,\pi]}$, and~$\de_{[\pi,\pi,\pi]}$. 

Take the signature of the generalized belongings of the type domain~$\Theta$ of the form $\Sigma_b\eq\scol{\eps_{\tau}}{\tau\in\Theta_b}=(\eps_{\kap},\eps_{\rho},\eps_{\lm})$.

Finally, take a denumerable set~$\Sigma_v^{\pi}$ of objective variables $x^{\pi},y^{\pi},\ldots$ of the first-order type~$\pi$ and denumerable sets~$\Sigma_v^{\kap}$, $\Sigma_v^{\rho}$, and~$\Sigma_v^{\lm}$ of predicate variables~$u^{\kap},v^{\kap},\ldots$, $u^{\rho},v^{\rho},\ldots$, and~$u^{\lm},v^{\lm},\ldots$ of the second-order types $\kap$, $\rho$, and~$\lm$, respectively.

They form the signature $\Sigma_v\equiv\scol{\Sigma_v^{\tau}}{\tau\in\Theta}=(\Sigma_v^{\pi},\Sigma_v^{\kap},\Sigma_v^{\rho},\Sigma_v^{\lm})$ of variables of the type domain~$\Theta$.

Consider the \emph{generalized signature} $\Sigma_{R2}^g\equiv\Sigma_c|\Sigma_e|\Sigma_b|\Sigma_v$ and the corresponding language~$L(\Sigma_{R2}^g)$. 
Terms $p,q,r,s,\ldots$ of this language are constants and variables only; the atomic equality formulas have the forms 
$q^{\pi}\depi r^{\pi}$, $q^{\kap}\dekap r^{\kap}$, $q^{\rho}\derho r^{\rho}$, and $q^{\lm}\delm r^{\lm}$.
Respectively,  the atomic belonging formulas have the forms 
$q^{\pi}\epskap r^{\kap}$, $(p^{\pi},q^{\pi})\epsrho r^{\rho}$, and $(p^{\pi},q^{\pi},r^{\pi})\epslm s^{\lm}$.

Further, along with $x^{\pi}$, $y^{\pi}$, and $\depi$ we shall simply write~$x$, $y$, and~$\dee$, respectively.

Along with the generalized signature $\Sigma_{R2}^g$ we consider the \emph{standard signature} $\Sigma_{R2}^{st}\equiv\Sigma_c|\Sigma_e^{st}|\Sigma_b^{st}|\Sigma_v$, where in the signature of the standard equalities $\Sigma_e^{st}\eq\scol{\detau^{st}}{\tau\in\Theta}$ the type equalities~$\detau^{st}$ are one and the same \emph{standard equality~$\dest$} and in the signature of the standard belongings $\Sigma_b^{st}\eq\scol{\epstau^{st}}{\tau\in\Theta_b}$ the type belongings~$\epstau^{st}$ are one and the same \emph{standard belonging~$\epst$}.

Respectively, this signature $\Sigma_{R2}^{st}$ generates the standard language~$L(\Sigma_{R2}^{st})$ with atomic equality formulas of the forms 
$q^{\pi}\dest r^{\pi}$, $q^{\kap}\dest r^{\kap}$, $q^{\rho}\dest r^{\rho}$, and $q^{\lm}\dest r^{\lm}$
and with atomic belonging formulas of the forms 
$q^{\pi}\epst r^{\kap}$, $(p^{\pi},q^{\pi})\epst r^{\rho}$, and $(p^{\pi},q^{\pi},r^{\pi})\epst s^{\lm}$ for all terms $p,q,r,s,\ldots$.

\subsection{The axiomatics for the generalized and the standard second-order Dedekind theories of real numbers}\label{sec-Ded-axiom}
\mbox{}

The signature $\Sigma^g_{R2}$ gives the opportunity to define the language~$L(\Sigma^g_{R2})$ and to construct the desired models of the generalized second-order theory of real numbers, but the absence of functional variables in this signature makes the writing of generalized axioms for this theory very unusual. Only the names of these axioms placed in round brackets clarify their customary sense.

The \emph{axioms of the generalized second-order Dedekind theory of real numbers} are the following ones.

\textbf{A1} (the \emph{existence and functionality of the negation}).
$$
(\forall\,x\ \exists\,y\ ((x,y)\epsrho-))\wedge(\forall\,x,y,y'\ ((x,y)\epsrho-)\wedge((x,y')\epsrho-)\impl y\dee y').
$$

\textbf{A2} (the \emph{existence and functionality of the addition}).
$$
(\forall\,x,y\ \exists\,z\ ((x,y,z)\epslm+))\wedge(\forall\,x,y,z,z'\ ((x,y,z)\epslm+)\wedge((x,y,z')\epslm+)\impl z\dee z').
$$

\textbf{A3} (the \emph{existence and functionality of the inversion}).
\begin{multline*}
(\forall\,x\ (\neg(x\dee0)\impl\exists\,y\ ((x,y)\epsrho/)))\wedge(\forall\,x,y\ (((x,y)\epsrho/)\impl\neg(x\dee0)))\wedge\\
\wedge(\forall\,x,y,y'\ ((x,y)\epsrho/)\wedge((x,y')\epsrho/)\impl y\dee y').
\end{multline*}

\textbf{A4} (the \emph{existence and functionality of the multiplication}).
$$
(\forall\,x,y\ \exists\,z\ ((x,y,z)\epslm\cdot))\wedge(\forall\,x,y,z,z'\ ((x,y,z)\epslm\cdot)\wedge((x,y,z')\epslm\cdot)\impl z\dee z').
$$

The appearance of axioms A1--A4 in this list is directly impelled by the absence of functional variables in the signature $\Sigma^g_{R2}$.

\textbf{A5} (the \emph{non-equality of the unit and the null}). $\neg(1\dee0)$.

\textbf{A6} (the \emph{associativity of the addition}).
\begin{multline*}
\forall\,x,y,z\ \forall\,u_1,u_2,v_1,v_2\ (((x,y,u_1)\epslm+)\wedge((u_1,z,u_2)\epslm+)\wedge\\
\wedge((y,z,v_1)\epslm+)\wedge ((x,v_1,v_2)\epslm+) \impl u_2\dee v_2).
\end{multline*}
The writing of axiom A6 in the common way:
$\forall\,x,y,z\ (((x+y)+z)\dee(x+(y+z)))$.

\textbf{A7} (the \emph{neutrality of the null}).
$$
\forall\,x\ \forall u,v\ ((((x,0,u)\epslm+)\impl u\dee x)\wedge(((0,x,v)\epslm+)\impl v\dee x)).
$$

\textbf{A8} (the \emph{elimination of the negation}).
\begin{multline*}
\forall\,x\ \forall\,u_1,u_2,v_1,v_2\ ((((x,u_1)\epsrho-)\wedge((x,u_1,u_2)\epslm+)\impl u_2\dee0)\wedge\\ \wedge(((x,v_1)\epsrho-)\wedge((v_1,x,v_2)\epslm+)\impl v_2\dee0)).
\end{multline*}

\textbf{A9} (the \emph{commutativity of the addition}).
$$
\forall\,x,y\ \forall\,u,v\ (((x,y,u)\epslm+)\wedge((y,x,v)\epslm+)\impl u\dee v).
$$

\textbf{A10} (the \emph{right distributivity of the multiplication with respect the addition}).
\begin{multline*}
\forall\,x,y,z\ \forall\,u_1,u_2,v_1,v_2,v_3\ (((y,z,u_1)\epslm+)\wedge((x,u_1,u_2)\epslm\cdot)\wedge\\
\wedge((x,y,v_1)\epslm\cdot)\wedge((x,z,v_2)\epslm\cdot)\wedge((v_1,v_2,v_3)\epslm+)\impl u_2\dee v_3).
\end{multline*}
The writing of this axiom in the common way:
$\forall\,x,y,z\ ((x\cdot(y+z))\dee(x\cdot y+x\cdot z))$.

\textbf{A11} (the \emph{left distributivity of the multiplication with respect the addition}).
\begin{multline*}
\forall\,x,y,z\ \forall\,u_1,u_2,v_1,v_2,v_3\ (((x,y,u_1)\epslm+)\wedge((u_1,z,u_2)\epslm\cdot)\wedge\\
\wedge((x,z,v_1)\epslm\cdot)\wedge((y,z,v_2)\epslm\cdot)\wedge((v_1,v_2,v_3)\epslm+)\impl u_2\dee v_3).
\end{multline*}

\textbf{A12} (the \emph{associativity of the multiplication}).
\begin{multline*}
\forall\,x,y,z\ \forall\,u_1,u_2,v_1,v_2\ (((x,y,u_1)\epslm\cdot)\wedge((u_1,z,u_2)\epslm\cdot)\wedge\\
\wedge((y,z,v_1)\epslm\cdot)\wedge ((x,v_1,v_2)\epslm\cdot) \impl u_2\dee v_2).
\end{multline*}

\textbf{A13} (the \emph{neutrality of the unit}).
$$
\forall\,x\ \forall u,v\ ((((x,1,u)\epslm\cdot)\impl u\dee x)\wedge(((1,x,v)\epslm\cdot)\impl v\dee x)).
$$

\textbf{A14} (the \emph{elimination of the inversion}).
\begin{multline*}
\forall\,x\ \forall\,u_1,u_2,v_1,v_2\ (\neg(x\dee0)\impl(((x,u_1)\epsrho/)\wedge\\
\wedge((x,u_1,u_2)\epslm\cdot)\impl u_2\dee1) \wedge(((x,v_1)\epsrho/)\wedge((v_1,x,v_2)\epslm\cdot)\impl v_2\dee1)).
\end{multline*}
The writing of A14 in the common way is the following:
$$
 \forall\,x\ (\neg(x\dee0)\impl(x\cdot(x^{-1})\dee1)\wedge((x^{-1})\cdot x\dee1).
$$

\textbf{A15} (the \emph{commutativity of the multiplication}).
$$
\forall\,x,y\ \forall\,u,v\ (((x,y,u)\epslm\cdot)\wedge((y,x,v)\epslm\cdot)\impl u\dee v).
$$

Further, along with $(x,y)\epsrho\le$ we shall write $x\le y$ as well. It gives the opportunity to write the subsequent axioms in a more customary form.

\textbf{A16} (the \emph{reflexivity of the order}). $\forall\,x\ (x\le x)$.

By E4 we get $x\dee y\impl(x\le x \iimpl x\le y)$. Applying A16, we conclude that $x\dee y \vdash x\le y$.

\textbf{A17} (the \emph{antisymmetry of the order}). $\forall\,x,y\ (((x\le y)\wedge (y\le x))\impl x\dee y)$.

\textbf{A18} (the \emph{transitivity of the order}). $\forall\,x,y,z\ (((x\le y)\wedge (y\le z))\impl x\le z)$.

\textbf{A19} (the \emph{linearity of the order}). $\forall\,x,y\ ((x\le y)\vee(y\le x))$.

\textbf{A20} (the \emph{compatibility of the addition and the order}). 
$$
\forall\,x,y,z\ \forall\,u,v\ (x\le y\impl(((x,z,u)\epslm+)\wedge((y,z,v)\epslm+)\impl u\le v)).
$$

\textbf{A21} (the \emph{compatibility of the multiplication and the order}). 
$$
\forall\,x,y\ \forall\,u\ ((x\ge0)\wedge(y\ge0)\impl(((x,y,u)\epslm\cdot)\impl u\ge0)).
$$

\textbf{A22} (the \emph{existence of Dedekind cuts}).
\begin{multline*}
\forall\,u^{\kap},v^{\kap}\ ((\exists\,x\ (x\epskap u^{\kap}))\wedge(\exists\,y\ (y\epskap v^{\kap}))\wedge\\
\wedge(\forall\,z\ ((z\epskap u^{\kap})\vee(z\epskap v^{\kap})))\wedge(\forall\,x,y\ ((x\epskap u^{\kap})\wedge(y\epskap v^{\kap})\impl x\le y))\impl\\
\impl(\exists\,z\ \forall\,x,y\ ((x\epskap u^{\kap})\wedge(y\epskap v^{\kap})\impl(x\le z)\wedge(z\le y)))).
\end{multline*}

Consider the following generalized \emph{extensionality properties}.

\textbf{PE1.} $\forall\, u^\varkappa,v^\varkappa\ (u^\varkappa\dekap v^\varkappa \Leftrightarrow \forall\, x\ (x\epskap u^\varkappa\Leftrightarrow x\epskap v^\varkappa)).$

\textbf{PE2.} $\forall\, u^\rho,v^\rho\ (u^\rho\derho v^\rho \Leftrightarrow \forall\, x,y\ ((x,y)\epsrho u^\rho \Leftrightarrow (x,y)\epsrho v^\rho)).$

\textbf{PE3.} $\forall\, u^\lm,v^\lm\ (u^\lm\delm v^\lm \Leftrightarrow \forall\, x,y,z\ ((x,y,z)\epslm u^\lm \Leftrightarrow (x,y,z)\epslm v^\lm)).$

The theory determined by the language $L(\Sigma_{R2}^g)$ and the set of axioms $\Psi_2^g\eq\{$E1--E4, A1--A22, PE1--PE3$\}$ can be called the \emph{generalized second-order Dedekind theory of real numbers}. It will be denoted by $Th_{R2}^g$.

Respectively, in the language $L(\Sigma_{R2}^{st})$ we can write formulas E1$^{st}$--E4$^{st}$, A1$^{st}$--A22$^{st}$, PE1$^{st}$--PE3$^{st}$, which are obtained from the corresponding formulas E1--E4, A1--A22, PE1--PE3 of the language $L(\Sigma_{R2}^g)$ by the substitution of the generalized type equalities and belongings~$\detau$ and~$\epstau$ by the standard ones~$\dest$ and~$\epst$, respectively.

The theory determined by the language $L(\Sigma_{R2}^{st})$ and axioms E1$^{st}$--E4$^{st}$, A1$^{st}$--A22$^{st}$, PE1$^{st}$--PE3$^{st}$ can be called the \emph{standard second-order Dedekind theory of real numbers}. It will be denoted by $Th_{R2}^{st}$.

\subsection{The canonical generalized and standard second-order Dedekind real axes}\label{sec-Ded-axes}
\mbox{}

Consider the canonical set~$\Rbb$ of all real numbers constructed in the considered set theory ST (see, e.\,g., \cite[1.4]{ZakhRodi2018SFM1} for NBG set theory and~\cite{Zakharov2005LTS} and~\cite[B.1]{ZakhRodi2018SFM1} for the LTS).

For the set~$\Rbb$ and the signature~$\Sigma_{R2}^g$ consider the collections 
\begin{align*}
S_c^{\pi}&\eq\col{s_{\omega}^{\pi}}{\omega\in\Omega_{\pi}}=\zzl s_0^{\pi},s_1^{\pi}\zzr,& 
S_c^{\kap}&\eq\col{s_{\omega}^{\kap}}{\omega\in\Omega_{\kap}}=\vrn,\\ 
S_c^{\rho}&\eq\col{s_{\omega}^{\rho}}{\omega\in\Omega_{\rho}}=\zzl s_0^{\rho},s_1^{\rho},s_2^{\rho}\zzr,&\text{ and }
S_c^{\lm}&\eq\col{s_{\omega}^{\lm}}{\omega\in\Omega_{\lm}}=\zzl s_0^{\lm},s_1^{\lm}\zzr.
\end{align*}
They compose the collection of constants structures
$$
 S_c=\col{S_c^{\tau}}{\tau\in\Theta}=
\big\zzl\zzl s_0^{\pi},s_1^{\pi}\zzr,\vrn,\zzl s_0^{\rho},s_1^{\rho},s_2^{\rho}\zzr,\zzl s_0^{\lm},s_1^{\lm}\zzr\big\zzr
$$
containing the constant structures $s_0^{\pi},s_1^{\pi}\in\pi(\Rbb)=\Rbb$ which are the \emph{neutral real numbers}, the constant structures~$s_0^{\rho},s_1^{\rho},s_2^{\rho}\in\rho(\Rbb)=\cP(\Rbb^2)$, which are the \emph{ratio of negation}, the \emph{ratio of inversion}, and the \emph{ratio of order on~$\Rbb$}, respectively, and the constant structures~$s_0^{\lm},s_1^{\lm}\in\lm(\Rbb)=\cP(\Rbb^3)$ which are the \emph{ratio of addition} and the \emph{ratio of multiplication on~$\Rbb$}, respectively.

Further, along with $s_0^{\pi}$, $s_1^{\pi}$, $s_0^{\rho}$, $s_1^{\rho}$, $s_2^{\rho}$, $s_0^{\lm}$, and $s_1^{\lm}$ we shall simply write~$0_{\Rbb}$, $1_{\Rbb}$, $-_{\Rbb}$, $/_{\Rbb}$, $\le_{\Rbb}$, $+_{\Rbb}$, and~$\cdot_{\Rbb}$, respectively.

Consider the collection of the equality ratios of the form 
\begin{multline*}
S_e\eq\col{\approx_{\tau}}{\tau\in\Theta}=\zzl\approx_{\pi},\approx_{\kap},\approx_{\rho},\approx_{\lm}\zzr\eq\\
\eq\zzl=|\Rbb^2,=|\cP(\Rbb)^2,=|\cP(\Rbb^2)^2,=|\cP(\Rbb^3)^2\zzr
\end{multline*} 
containing in the capacity of the first-order equality ratio~$\approx_{\pi}$ and of the second-order equality ratios~$\approx_{\kap}$, $\approx_{\rho}$, and~$\approx_{\lm}$ the restrictions on the indicated sets one and the same set-theoretical equality in ST.

Consider the collection of the belonging ratios of the form 
\begin{multline*}
 S_b\eq\col{\inn_{\tau}}{\tau\in\Theta}=
\zzl\inn_{\kap},\inn_{\rho},\inn_{\lm}\zzr\eq\\
\eq\zzl\in|(\Rbb\times\cP(\Rbb)),\in|(\Rbb^2\times\cP(\Rbb^2)),\in|(\Rbb^3\times\cP(\Rbb^3))\zzr
\end{multline*}
containing in the capacity of the belonging ratios~$\inn_{\kap}$, $\inn_{\rho}$, 
and~$\inn_{\lm}$ the restrictions on the indicated sets one and the same set-theoretical belonging ratio~$\in$ in ST.

Finally, take the collection of the terminals over the set~$\Rbb$ of the form 
$$
S_v\eq\col{\tau(\Rbb)}{\tau\in\Theta}=\zzl\pi(\Rbb),\kap(\Rbb),\rho(\Rbb),\lm(\Rbb)\zzr=
\zzl\Rbb,\cP(\Rbb),\cP(\Rbb^2),\cP(\Rbb^3)\zzr.
$$

These collections compose the superstructure $S_{R2}\eq\zzl S_c,S_e,S_b,S_v\zzr$ of the signature $\Sigma_{R2}^g$.
The system $\zzl\Rbb,S_{R2}\zzr$ of the signature $\Sigma_{R2}^g$ can be called the \emph{canonical generalized second-order Dedekind real axis in ST}. It will be denoted by $R_2^g$.

Consider an evaluation $\zeta$ on the system $R_2^g$ such that $\zeta(x)\in\pi(\Rbb)=\Rbb$, $\zeta(u^\kap)\in\kap(\Rbb)=\cP(\Rbb)$, $\zeta(u^\rho)\in\rho(\Rbb)=\cP(\Rbb^2)$, and $\zeta(u^\lm)\in\lm(\Rbb)=\cP(\Rbb^3)$.
Thus, we get the evaluated system $\zzl R_2^g,\zeta\zzr$.

The above constructed superstructure $S_{R2}$ is also the superstructure of the signature $\Sigma_{R2}^{st}$. Therefore the system $\zzl\Rbb,S_{R2}\zzr$ is also the system of the signature $\Sigma_{R2}^{st}$. It can be called the \emph{canonical standard second-order Dedekind real axis in ST}. It will be denoted by $R_2^{st}$.

The evaluation $\zeta$ on the system $R_2^g$ considered above is also an evaluation on the system $R_2^{st}$. Therefore we may consider the evaluated system $\zzl R_2^{st},\zeta\zzr$.

Let $B$ be a set and $T_{R2}^{st}$ be a superstructure on $B$ of the signature $\Sigma_{R2}^{st}$.
Consider the system $V\eq\zzl B,T_{R2}^{st}\zzr$ and some evaluation~$\eta$ on~$V$.
For the evaluated system $\zzl V,\eta\zzr$ we shall use the following designations:
$0_B\eq\si_0^{\pi}[\eta]$, $1_B\eq\si_1^{\pi}[\eta]$, $-_B\eq\si_0^{\rho}[\eta]$, $/_B\eq\si_1^{\rho}[\eta]$, $\le_B\eq\si_2^{\rho}[\eta]$, $+_B\eq\si_0^{\lm}[\eta]$, and $\cdot_B\eq\si_1^{\lm}[\eta]$.

The (\emph{standard}) \emph{satisfaction $U\vDash_{st}\vphi[\eta]$ of a formula~$\vphi$ of the language $L(\Sigma^{st}_{R2})$ on the system~$V$ of the signature~$\Sigma^{st}_{R2}$ with respect to the evaluation~$\eta$} differs from the (generalized) satisfaction from~\ref{sec-systSigma-eval} only in the first two points:
\begin{enumerate}
\item[$1'.$] 
 if $q$ and $r$ are terms of a type $\tau\in\Theta$ and $\vphi\eq(q\dest r)$, then $V\vDash_{st}\vphi[\eta]$ is equivalent to $q[\eta]=r[\eta]$;
\item[$2'.$] 
 if $\tau_0,\ldots,\tau_k$ are types from $\Theta$ for $k\ge0$, $\tau\equiv[\tau_0,\ldots,\tau_k]\in\Theta$, $q_0,\ldots,q_k$ are terms of the types $\tau_0$, \ldots,$\tau_k$, respectively, $r$ is a term of the type $\tau$, and $\varphi\equiv(q_0,\ldots,q_k)\epst r$, then $U\vDash\varphi[\eta]$ iff $(q_0[\eta],\ldots,q_k[\eta])\inn r[\eta]$.
\end{enumerate}

Let $\Phi$ be a set of formulas of the language $L(\Sigma^{st}_{R2})$. As in~\ref{sec-systSigma-satisf} the evaluated system $\zzl V,\eta\zzr$ of the signature~$\Sigma^{st}_{R2}$ is called a \emph{standard model for the set~$\Phi$} if $V\vDash_{st}\vphi[\eta]$ for every~$\vphi\in\Phi$.

Now we can formulate some initial theorem about the standard~$R_2^{st}$ and the generalized~$R_2^g$ Dedekind real axes.

\begin{theo}\label{theo-model-Ded-axes}
	\mbox{}
	\begin{enumerate}
		\item
		The mathematical system $R_2^{st}$ is a standard model for the theory $Th_{R2}^{st}$ in the language $L(\Sigma_{R2}^{st})$.
		\item
		The mathematical system $R_2^g$ is a \textup{(\emph{generalized})} model for the theory $Th_{R2}^g$ in the language $L(\Sigma_{R2}^g)$.
	\end{enumerate}
\end{theo}

\begin{proof}
	1.
	Note that all the axioms from the set $\Psi_2^g$ are closed formulas. Therefore the satisfaction $R_2^{st}\vDash \al[\zeta]$ for $\al\in\Psi_2^g$ means the deducibility of the relativization~$\al^r$ of~$\al$ on~$\Rbb$ in the considered axiomatic set theory~ST. But the corresponding deducibility of every~$\al^r$ is very well demonstrated in mathematical literature (see, for example, \cite{Bourbaki(III3-4)1960,Fef1963,Landau1958,HewStr1965,GrLiFi1968,ZakhRodi2018SFM1}).
	
	2.
	This assertion follows directly from assertion~1 by virtue of the inclusions $=|\tau(\Rbb)\times\tau(\Rbb)\subset\approx_{\tau}$ and $\in|\check{\tau}(\Rbb)\times\tau(\Rbb)\subset\inn_{\tau}$ from~\ref{sec-systSigma-eval}, where the left parts of the inclusions are the restrictions of the usual set-theoretical ratios~$=$ and~$\in$ on the indicated sets. 
\end{proof}

The models from Theorem~\ref{theo-model-Ded-axes} are called \emph{canonical}.

It is well known that the theory $Th_{R2}^{st}$ is categorical. On the contrary, we shall prove that the theory~$Th_{R2}^g$ is non-categorical. More exactly, using the initial canonical model~$R_2^g$ with the support~$\Rbb$ we shall prove the existence of some non-canonical models for the theory~$Th_{R2}^g$ having arbitrary large powers.

This statement can be proven with the help of the generalized infrafiltration theorem (see, e.\,g., \cite{ZakhYash2014} or~\cite[C.3.2]{ZakhRodi2018SFM1}). But to make the paper self-contained we prefer to prove here some more simple variant of the generalized infrafiltration theorem than it is presented in the indicated works.

\section{The infraproduct construction of evaluated systems of the signature $\Sigma_2^g$}\label{sec-main}

\subsection{Infraproducts of collections of evaluated systems of the signature $\Sigma_2^g$}\label{sec-infra-prod}
\mbox{}

Let $F$ be a fixed set and $\col{U_f}{f\in F}$ be a fixed collection of mathematical systems of the signature $\Sigma_2^g$ with true generalized equalities and belongings.

By definition, $U_f\equiv\zzl A_f,S_f\zzr$. Consider the set $A\equiv\prod\col{A_{f}}{f\in F}$.

Let $\tau\equiv[\tau_0,\ldots,\tau_k]$ be a second-order type and $k\ge0$.
If $\mu\in k+1$, then $\tau_\mu=0$. Thus, we see that $\tau_\mu(A)=A=\prod\col{A_{f}}{f\in F}=\prod\col{\tau_\mu(A_f)}{f\in F}$.
For elements $p\in\check\tau(A)\eq\tau_0(A)\times\cdots\times\tau_k(A)=A^{k+1}$ and $f\in F$ define the element $p(f)\in\check\tau(A_f)=\tau_0(A_f)\times\cdots\times\tau_k(A_f)=A_f^{k+1}$ setting $p(f)(\mu)\equiv p(\mu)(f)$ 
for every $\mu\in k+1$.

For elements $P\subset\check\tau(A)$ and $f\in F$ define the element $P\langle f\rangle\subset\check\tau(A_f)$ setting 
$P\langle f\rangle\equiv\{\xi\in\check\tau(A_f)\mid \exists\, p\in P\ (p(f)=\xi)\}$.

Let $\cD$ be a subset of the set $\cP(F)$, i.\,e., an ensemble on~$F$. Define some superstructure $S$ of the signature $\Sigma_2^g$ over the set~$A$.

First, define constant structures $s_\omega^\tau\in\tau(A)$ for $\tau\in\Theta$ and $\omega\in\Omega_\tau$.

If $\tau$ is a first-order type, then $\tau(A)=\prod\col{\tau(A_f)}{f\in F}$. Therefore define $s_\omega^\tau\in\tau(A)$ setting $s_\omega^\tau(f)\equiv s_{\omega f}^\tau$ for every $f\in F$.

Put $s_{\omega}^\tau\equiv\{p\in\check\tau(A)\mid\forall\, f\in F\ (p(f)\in s_{\omega f}^\tau)\}$ if $\tau\eq[\tau_0,\ldots,\tau_k]$ is a second-order type.

As a result, we obtain the collections $S_c^\tau\equiv\col{s_\omega^\tau}{\omega \in\Omega_\tau}$ and the collection $S_c\equiv\col{S_c^\tau}{\tau\in\Theta}$.

Now define generalized equality ratios $\approx_\tau\subset\tau(A)\times\tau(A)$.
If $\tau$ is the first-order type, then for $p,q\in\tau(A)$ put $p\approx_\tau q$ if
$\exists\,G\in\cD\ \forall\,g\in G\ (p(g)\approx_{\tau,g}q(g))$.

If $\tau\eq[\tau_0,\ldots,\tau_k]$ is a second-order type, then for $P,Q\subset\check\tau(A)$ put $P\approx_\tau Q$ if
$\exists\,G\in\cD\ \forall\, g\in G\ (P\langle g\rangle\approx_{\tau,g} Q\langle g\rangle)$.

As a result, we obtain the collection $S_e\equiv\col{\approx_\tau}{\tau\in\Theta}$.

Now define generalized belonging ratios $\inn_\tau\subset \check{\tau}(A)\times\tau(A)$. 

By definition, $\tau=[\tau_0,\ldots,\tau_k]$ for some $\tau_0,\ldots,\tau_k\in\Theta$. For $p\in\check\tau(A)$ and $P\subset\check\tau(A)$ put $p\inn_\tau P$ if
$\exists\,G\in\cD\ \forall\, g\in G\ (p(g) \inn_{\tau,g}P\langle g\rangle)$. Note that the usage of a generalized belonging ratio was explored in the forcing method in the form $x\in_p y$ (see, e.\,g., \cite[9.8]{Shoenfield2001}).

Thus, we obtain the collection $S_b\equiv\col{\inn_\tau}{\tau\in\Theta_b}$.

Consider also the collection $S_v\equiv\col{\tau(A)}{\tau\in\Theta}$ consisting of the $\tau$-terminals of the set~$A$.

The constructed collections compose the superstructure $S\equiv\zzl S_c,S_e,S_b,S_v\zzr$ over the set~$A$. Therefore we can consider the mathematical system $U\equiv\zzl A,S\zzr$ of the signature~$\Sigma_2^g$. It will be called the \emph{infra-{$\cD$}-product of the collection of mathematical systems $\col{U_f}{f\in F}$ of the generalized second-order signature} $\Sigma_2^g$ and will be denoted  by $\infraDprod\col{U_f}{f\in F}$.

An ensemble~$D$ on~$F$ is called a \emph{filter on~$F$} if it has the following properties:
\begin{enumerate}
	\item $\forall\,G,H\in\cD\ (G\cap H\in\cD)$;
	\item $\forall\,G\in\cD\ \forall\,H\in\cP(F)\ (G\subset H\impl H\in\cD)$.
\end{enumerate}

A filter~$\cD$ is called \emph{proper} if $\cD\ne\cP(F)$. A proper filter~$\cD$ is called an \emph{ultrafilter} if for any proper filter~$\cE$ on~$F$ such that $\cD\subset\cE$ we have $\cD=\cE$, i.\,e., $\cD$ is a maximal element in the set of all proper filters on~$F$.

A pair $\zzl G,H\zzr$ of subsets of $F$ is called a \emph{binary partition of~$F$} if $G\cap H=\vrn$ and $G\cup H=F$. 
A filter~$\cD$ is a ultrafilter iff it has the \emph{binary partition property}, i.\,e., if for every binary partition 
$\zzl G,H\zzr$ of~$F$ either $G\in\cD$ or $H\in\cD$ (see~\cite[Exercise~2.119]{Mendelson1997}). 

Further on, we assume that $\cal D$ is a filter.

Now let $\col{\zzl U_f,\gamma_f\zzr}{f\in F}$ be a collection of evaluated mathematical systems of the second-order signature $\Sigma_2^g$  with true generalized equalities and belongings.

Define an evaluation $\gamma$  on the system $U\equiv \infraDprod\col{U_f}{f\in F}$ in the following way.

Let $x$ be a variable of a type $\tau$. If $\tau$ is the first-order type, then define $\gamma(x)\in\tau(A)$ setting $\gamma(x)(f)\equiv\gamma_f(x)$ for every $f\in F$. If $\tau=[\tau_0,\ldots,\tau_k]$ is a second-order type, then put $\gamma(x)\equiv\{p\in\check\tau(A)\mid\forall f\in F\ (p(f)\in\gamma_f(x))\}$.

The evaluation $\gamma$ will be called the \emph{crossing of the collection of evaluations} $\col{\gamma_f}{f\in F}$ and will be denoted by $\bowtie\col{\gamma_f}{f\in F}$.

\begin{lem}\label{lem-2-infra-prod}
	Let $\col{\lgroup U_f,\gamma_f\rgroup}{f\in F}$ be a collection of evaluated mathematical systems of the second-order signature $\Sigma_2^g$ and let every system $\zzl U_f,\gamma_f\zzr$ be a model for equality axioms \emph{E1--E4}. Then the pair $\zzl\infraDprod\col{U_f}{f\in F}, \bowtie\col{\gamma_f}{f\in F}\zzr$ is also a model for axioms \emph{E1--E4}.
\end{lem}

\begin{proof}
	Let $t_0,t_0'\in\tau_0(A)$, \ldots, $t_k,t_k'\in\tau_k(A)$, $P,P'\subset\check\tau(A)=\tau_0(A)\times\ldots\times\tau_k(A)$, $p\equiv(t_0,\ldots,t_k)$, $p'\equiv(t_0',\ldots,t_k')$, $p\approx_{\check\tau}p'$, and $P\approx_\tau P'$.
	
	Assume that $p\inn_\tau P$. According to the definition of the belonging, we get
	$\exists\,G_1\in\cD\ \forall\,g\in G_1\ (p(g)\inn_{\tau,g}P\langle g\rangle)$. By the definition of the first-order equality, $\exists\, G_2\in\cD\ \forall\, g\in G_2\ (p(g)\approx_{\check\tau,g}p'(g))$. Finally, by the definition of the second-order equalities $\exists\, G_3\in\cD\ \forall\,g\in G_3\ (P\langle g\rangle\approx_{\tau,g}P'\langle g\rangle)$. 
	Since every system $\zzl U_g,\gamma_g\zzr$ satisfies~E4, we see that $p'(g)\inn_{\tau,g}P'\langle g\rangle$ for every 
	$g\in G\equiv G_1\cap G_2\cap G_3$. Thus, $p'\inn_\tau P'$. Hence, $p\inn_\tau P\Rightarrow p'\inn_\tau P'$. The inverse implication is checked quite similarly. This proves axiom~E4. The validity of axioms E1, E2, E3 is obvious.
\end{proof}

Further, for a formula $\varphi\in L(\Sigma)$ the set $\{f\in F\mid U_f\vDash\varphi[\gamma_f]\}$ will be denoted by $G_\varphi$.

\begin{lem}\label{lem-3-infra-prod}
	Let $\tau=[\tau_0,\ldots,\tau_k]$ be a second-order type. Let $s_\omega^\tau$ be the constants constructed above for the support $A\equiv\prod\col{A_f}{f\in F}$. Then $s_\omega^\tau\langle f\rangle=s_{\omega f}^\tau$ for every $f\in F$.
\end{lem}

\begin{proof}
	Let $\xi\in s_\omega^\tau\langle f\rangle$, i.\,e., $\xi=p(f)$ for some $p\in s_\omega^\tau$. 
	By definition, $\xi=p(f)\in s_{\omega f}^\tau$. Consequently, $s_\omega^\tau\langle f\rangle\subset s_{\omega f}^\tau$.
	
	Conversely, let $\xi_f\in s_{\omega f}^\tau$. Using the axiom of choice we can find a collection 
	$\scol{\xi_g}{g\in F\setminus\{f\}}$ such that $\xi_g\in s_{\omega g}^\tau$. Define the element $p\in\check\tau(A)$ setting $p(\mu)(g)\equiv\xi_g(\mu)$ for every $g\in F$ and every $\mu\in k+1$. Then $p(g)=\xi_g\in s_{\omega g}^\tau$ for every $g\in F$ implies $p\in s_{\omega}^\tau$. Since $\xi_f=p(f)$, we have $\xi_f\in s_\omega^\tau\langle f\rangle$. 
	Hence, $s_{\omega f}^\tau\subset s_\omega^\tau\langle f\rangle$.
\end{proof}

\begin{lem}\label{lem-4-infra-prod}
	Let $\tau=[\tau_0,\ldots,\tau_k]$ be a second-order type. Let $x$ be a variable of the type~$\tau$ and $\gamma(x)$ be the evaluation constructed above for the system $U\equiv\zzl A,S\zzr$. Then $\gamma(x)\langle f\rangle=\gamma_f(x)$ for every $f\in F$.
\end{lem}

The proof is completely similar to the proof of the previous lemma.

\subsection{Infrafilteration of formulas of the second-order language $L(\Sigma_2^g)$}\label{sec-infra-filt}
\mbox{}

Consider a non-empty set $F$ and a filter $\cD$ on $F$.

By analogy with the first order language (see~\cite[\S\,17]{ErshPal1984}, \cite[8.2]{Maltsev1973}) a formula $\varphi$ of the language $L(\Sigma_2^g)$ of the second-order signature $\Sigma_2^g$ with generalized  equalities and belongings will be called \emph{infrafiltrated with respect to the filter $\cD$} if for every collection $\col{\zzl U_f,\gamma_f\zzr}{f\in F}$ of evaluated mathematical systems of the second-order signature $\Sigma_2^g$  with true generalized equalities and belongings the property $\infraDprod\col{U_f}{f\in F}\vDash \varphi[\bowtie\col{\gamma_f}{f\in F}]$ is equivalent to the property $\{g\in F\mid U_g\vDash\varphi[\gamma_g]\}\in\cD$.

\begin{lem}\label{lem-5-infra-filt}
	Every atomic formula is infrafiltrated with respect to any filter $\cD$ on the set~$F$.
\end{lem}

\begin{proof}
	First, consider an atomic formula $\varphi$ of the form $q^\tau\delta_\tau r^\tau$. Assume that $q^\tau=x^\tau$ and $r^\tau=\sigma_\omega^\tau$. Then $U\vDash\varphi[\gamma]$ is equivalent to $\gamma(x)\approx_\tau s_\omega^\tau$, and analogously for the pair $\zzl U_f,\gamma_f\zzr$.
	
	Let $\tau$ be the first-order type. Let $G_\varphi\in\cD$, i.\,e., $\gamma_g(x)\approx_{\tau,g}s_{\omega g}^\tau$ for every 
	$g\in G_\varphi\in\cD$. Then $\gamma_g(x)=\gamma(x)$ and $s_{\omega g}^\tau=s_\omega^\tau(g)$ implies $\gm(x)(g)\approx_{\tau,g}s_{\omega}^{\tau}(g)$ for every $g\in G_\varphi\in\cD$. Thus, $\gamma(x)\approx_\tau s_\omega^\tau$, i.\,e., $U\vDash\varphi[\gamma]$.
	
	Conversely, let $U\vDash\varphi[\gamma]$, i.\,e., $\gamma(x)\approx_\tau s_\omega^\tau$. Then there exists $G\in\cD$ such that $\gamma(x)(g)\approx_{\tau,g}s_\omega^\tau(g)$ for every $g\in G$. But it means that 
	$\gamma_g(x)\approx_{\tau,g}s_{\omega g}^\tau$, i.\,e., $U_g\vDash\varphi[\gamma_g]$ for every $g\in G\in\cD$. 
	Since $G\subset G_\varphi$, we have $G_\varphi\in\cal D$.
	
	Now let $\tau\equiv[\tau_0,\ldots,\tau_k]$ be a second-order type. Let $G_\varphi\in\cD$, i.\,e., $\gamma_g(x)\approx_{\tau,g}s_{\omega g}^\tau$ for every $g\in G_\varphi\in\cD$.  According to Lemmas~\ref{lem-3-infra-prod} and~\mref{lem-4}{infra-prod}, the equalities $s_{\omega g}^\tau=s_\omega^\tau\langle g\rangle$ and 
	$\gamma_g(x)=\gamma(x)\langle g\rangle$ are correct. 
	Therefore $\gamma(x)\langle g\rangle\approx_{\tau,g}s_\omega^\tau\langle g\rangle$ for every $g\in G_\varphi\in\cD$.
	Consequently, $\gamma(x)\approx_\tau s_\omega^\tau$, i.\,e., $U\vDash\varphi[\gamma]$.
	
	Conversely, let $U\vDash\varphi[\gamma]$, i.\,e., $\gamma(x)\approx_\tau s_\omega^\tau$. By the definition of the second-order equality, $\gamma(x)\langle g\rangle\approx_{\tau,g} s_\omega^\tau\langle g\rangle$ for some $G\in\cD$ and every $g\in G$. Using Lemmas~\ref{lem-3-infra-prod} and~\mref{lem-4}{infra-prod} we obtain $\gamma_g(x)\approx_{\tau,g}s_{\omega g}^\tau$, i.\,e., $U_g\vDash\varphi[\gamma_g]$ for every $g\in G\in\cD$. Since $G\subset G_\varphi$, we infer that $G_\varphi\in\cD$.
	
	For the terms $q^\tau$ and $r^\tau$ of other forms the reasons are quite similar.
	
	Now consider an atomic formula $\varphi$ of the form $(q_0^{\tau_0},\ldots,q_k^{\tau_k})\epstau r^\tau$ for $\tau\equiv[\tau_0,\ldots,\tau_k]\in\Theta_b$. Assume that $q_\lambda^{\tau_\lambda}=x_\lambda^{\tau\lambda}$ and $r^\tau=u^\tau$ for some variables~$x_\lambda$ and~$u$. Then $U\vDash\varphi[\gamma]$ is equivalent to $(\gamma(x_0),\ldots,\gamma(x_k))\inn_\tau\gamma(u)$ and analogously for the pair $\zzl U_f,\gamma_f\zzr$.
	
	Let $G_\varphi\in\cD$, i.\,e., $(\gamma_g(x_0),\ldots,\gamma_g(x_k))\inn_{\tau,g}\gamma_g(u)$ for every $g\in G_{\varphi}$. Consider the elements $\xi_f\equiv(\gamma_f(x_0),\ldots,\gamma_f(x_k))$ and $p\equiv(\gamma(x_0),\ldots,\gamma(x_k))\in\check{\tau}(A)$. 
	Let $f\in F$. Then $p(f)(\mu)\equiv p(\mu)(f)=\gamma(x_\mu)(f)=\gamma_f(x_\mu)=\xi_f(\mu)$ for every $\mu\in k+1$. Consequently, $p(f)=\xi_f$. By Lemma~\mref{lem-3}{infra-prod} $\gamma_f(u)=\gamma(u)\langle f\rangle$. As a result, we obtain $p(g)\inn_{\tau,g}\gamma(x)\langle g\rangle$ for every $g\in G_\varphi\in\cD$. By definition, it means that $p\inn_\tau\gamma(x)$, i.\,e., $U\vDash\varphi[\gamma]$.
	
	Conversely, let $U\vDash\varphi[\gamma]$, i.\,e., $(\gamma(x_0),\ldots,\gamma(x_k))\inn_\tau\gamma(u)$. By the definition of the second-order belonging, for $p\equiv(\gamma(x_0),\ldots,\gamma(x_k))$ there exists $G\in\cD$ such that $p(g)\inn_{\tau,g}\gamma(u)\langle g\rangle$ for every $g\in G$. By Lemma~\mref{lem-4}{infra-prod} 
	$\gamma(u)\langle g\rangle=\gamma_g(u)$. By the previous subsection, $\xi_g=p(g)$. Consequently, $\xi_g\inn_{\tau,g}\gamma_g(u)$, i.\,e., $U_g\vDash\varphi[\gamma_g]$ for every $g\in G$. Since $G\subset G_\varphi$, we infer that $G_\varphi\in\cD$.
	
	For the terms $q_\lambda^{\tau_\lambda}$ and $r^\tau$  of other forms the reasons are quite similar.
\end{proof}

A proof of the property of infrafiltration for the quantified formula $\exists x^\tau\varphi$ for the language $L(\Sigma_2^g)$ of the generalized second-order signature $\Sigma_2^g$ is more delicate than for the first-order language. Therefore we begin it with a subsidiary proposition.

Let $\col{\zzl U_f,\gamma_f\zzr}{f\in F}$ be a collection of evaluated mathematical systems of the second-order signature $\Sigma_2^g$ with true generalized equalities and belongings. Let $\beta$ be an evaluation on the system $U\equiv\infraDprod\col{U_f}{f\in F}$.

For the evaluation $\beta$ and for every $f\in F$ define the evaluation $\delta_f$ on the system $U_f$ in the following way. Let $x$ be a variable of a type~$\tau$.
If $\tau$ is the first-order type, then put $\delta_f(x)\equiv\beta(x)(f)$.
If $\tau$  is a second-order type, then put $\delta_f(x)\equiv\beta(x)\langle f\rangle$.
Consider the evaluation $\delta\equiv\;\bowtie\!\col{\delta_f}{f\in F}$.

\begin{prop}\label{prop-4-infra-filt}
	The equalities $\delta(x^\tau)\approx_\tau\beta(x^\tau)$ hold for any variable~$x^\tau$.
\end{prop}

\begin{proof}
	If $\tau$ is the first-order type, then by the definition of the evaluations $\delta$ and $\delta_f$ we obtain $\delta(x)(f)\equiv\delta_f(x)=\beta(x)(f)$ for any $f\in F$, i.\,e., $\delta(x)=\beta(x)$.
	
	Let $\tau$ be a second-order type. Lemma~\mref{lem-4}{infra-prod} implies 
	$\delta(x)\langle f\rangle=\delta_f(x)=\beta(x)\langle f\rangle$ for any $f\in F$. By the definition of the second-order equality, we conclude that $\delta(x)\approx_\tau \beta(x)$.
\end{proof}

\begin{prop}\label{prop-5-infra-filt}
	Let a formula $\psi$ be infrafiltrated with respect to the filter $\cD$. Then the formula $\exists x^\tau\psi$ is infrafiltrated with respect to $\cD$ as well.
\end{prop}

\begin{proof}
	Denote the formula $\exists x^\tau\psi$ by $\varphi$. Let $G_\varphi\in\cD$, i.\,e., $U_g\vDash\varphi[\gamma_g]$ for every 
	$g\in G_\varphi\in\cD$. Further, we shall write simply~$G$ instead of~$G_\varphi$.
	
	The presented satisfaction property means that $U_g\vDash\psi[\gamma_g']$ for some evaluation $\gamma_g'$ such that $\gamma_g'(y)=\gamma_g(y)$ for every $y^\sigma\ne x^\tau$. For every $f\in F$ define the evaluation $\delta_f$ setting $\delta_f\equiv \gamma_f$ if $f\in F\setminus G$ and $\delta_f\equiv\gamma_f'$ if $f\in G$. 
	
	Check that the evaluated systems $\zzl U_f,\delta_f\zzr$ and $\zzl U_g,\delta_g\zzr$ are $H$-concordant for every $f,g\in F$. If $f,g\in F\setminus G$, then $\delta_f=\gamma_f$ and $\delta_g=\gamma_g$. Since the evaluations $\gamma_f$ and $\gamma_g$ are $H$-concordant, our assertion is true. Let $f,g\in G$. Then $\delta_f=\gamma_f'$ and $\delta_g=\gamma_g'$. Let $x$ be a variable of a type $\tau$.
	
	Consider the evaluation $\delta\equiv\;\bowtie\!\col{\delta_f}{f\in F}$.
	Check that $\delta(y)=\gamma(y)$ for every $y^\sigma\ne x^\tau$. 
	
	Let~$\sigma$ be the first-order type. Then $\delta(y)(g)=\delta_g(y)=\gamma_g'(y)=\gamma_g(y)=\gamma(y)(g)$ for $g\in G$. If $f\in F\setminus G$, then $\delta(y)(f)=\delta_f(y)=\gamma_f(y)=\gamma(y)(f)$. Consequently, $\delta(y)=\gamma(y)$.
	
	Let $\sigma$ be a second-order type. If $f\in G$, then $\delta_f(y)=\gamma_f'(y)=\gamma_f(y)$. If $f\in F\setminus G$, then $\delta_f(y)=\gamma_f(y)$. Let $p\in\delta(y)$. By the definition of the crossing, $p(f)\in\delta_f(y)$ for every $f\in F$. By the above, $p(f)\in\gamma_f(y)$ for every $f\in F$. This means that $p\in\gamma(y)$, whence $\delta(y)\subset\gamma(y)$. The inverse inclusion is checked in the same way. Consequently, $\delta(y)=\gamma(y)$.
	
	Thus, for every $y\ne x$ we have $\delta(y)=\gamma(y)$.
	
	By condition and construction, $U_g\vDash\psi[\delta_g]$ for every $g\in G\in\cD$. Since the formula $\psi$ is infrafiltrated, the  obtained property implies the property $U\vDash\psi[\delta]$. Since $\delta(y^\sigma)=\gamma(y^\sigma)$ for every $y^\sigma\ne x^\tau$, we obtain the property $U\vDash\varphi[\gamma]$.
	
	Conversely, let $U\vDash\varphi[\gamma]$. It is equivalent to $U\vDash\psi[\beta]$ for some evaluation $\beta$,  $H$-concordant with the evaluation $\gamma$ and such that $\beta(y)=\gamma(y)$ for every $y^\sigma\ne x^\tau$.
	
	Consider the evaluation $\delta\equiv\;\bowtie\!\col{\delta_f}{f\in F}$ from Proposition~\ref{prop-4-infra-filt}, corresponding to the evaluation $\beta$. According to Proposition~\ref{prop-4-infra-filt}, $\delta(z^\rho)\approx_\rho\beta(z^\rho)$ for every variable~$z^\rho$. It follows from Proposition~\mref{prop-2}{systSigma-satisf} that the property $U\vDash\psi[\beta]$ is equivalent to the property $U\vDash\psi[\delta]$. Since the formula $\psi$ is infrafiltrated, the property $U\vDash\psi[\delta]$ is equivalent to the property $G\equiv\{g\in F\mid U_g\vDash\psi[\delta_g]\}\in\cD$.
	
	Let $y^\sigma\ne x^\tau$. If $\sigma$ is the first-order type, then $\delta_g(y)=\beta(y)(g)=\gamma(y)(g)=\gamma_g(y)$.
	If~$\sigma$ is a second-order type, then $\delta_g(y)=\beta(y)\langle g\rangle=\gamma(y)\langle g\rangle$. Since by Lemma~\mref{lem-4}{infra-prod} $\gamma(y)\langle g\rangle=\gamma_g(y)$, we have $\delta_g(y)=\gamma_g(y)$. Consequently, in all the cases $\delta_g(y)=\gamma_g(y)$ for every $y^\sigma\ne x^\tau$. Therefore the property $U_g\vDash\psi[\delta_g]$ is equivalent to the property $U_g\vDash\varphi[\gamma_g]$. Thus, $\{g\in F\mid U_g\vDash\varphi[\gamma_g]\}=G\in\cD$.
	This implies $G_\varphi\in\cD$.
\end{proof}

The following two lemmas are the same as ones for the first-order language.

\begin{lem}\label{lem-6-infra-filt}
	Let formulas $\psi$ and $\xi$ be infrafiltrated with respect to the filter $\cD$. Then the formula $\psi\land\xi$ is infrafiltrated with respect to $\cD$ as well.
\end{lem}

\begin{proof}
	Denote the formula $\psi\land\xi$ by $\varphi$. Let $G_\varphi\in\cal D$, i.\,e., $U_g\vDash\varphi[\gamma_g]$ for all
	$g\in G_\varphi\in\cD$. This property is equivalent to the conjunction of the properties $U_g\vDash\psi[\gamma_g]$ and $U_g\vDash\xi[\gamma_g]$. Since these formulas are infrafiltrated, it is equivalent to the conjunction of the properties $U\vDash\psi[\gamma]$ and $U\vDash\xi[\gamma]$, but it is equivalent to the property $U\vDash\varphi[\gamma]$.
	
	Conversely, let $U\vDash\varphi[\gamma]$. It is equivalent to the conjunction of the properties $U\vDash\psi[\gamma]$ and $U\vDash\xi[\gamma]$. Then $G_\psi\in\cD$ and $G_\xi\in\cD$. Consider $G\equiv G_\psi\cap G_\xi$. Then $U_g\vDash\psi[\gamma_g]$ and $U_g\vDash\xi[\gamma_g]$ implies $U_g\vDash\varphi[\gamma_g]$ for every $g\in G\in\cD$. Hence, $G_\varphi\in\cD$.
\end{proof}

\begin{lem}\label{lem-7-infra-filt}
	Let a formula $\psi$ be infrafiltrated with respect to the ultrafilter $\cD$. Then the formula $\lnot\psi$ is infrafiltrated with respect to $\cD$ as well.
\end{lem}

\begin{proof}
	Denote the formula $\lnot\psi$ by $\varphi$. By assumption, the properties $G_\psi\in\cD$ and $U\vDash\psi[\gamma]$ are equivalent.
	
	By definition, $F\setminus G_\varphi=\{g\in F\mid \text{ the property } U_g\vDash\varphi[\gamma_g]\text{ does not hold}\}$. But $U_g\vDash\varphi[\gamma_g]$ is equivalent to the assertion that the property $U_g\vDash\psi[\gamma_g]$ does not hold. Consequently the property $U_g\vDash\psi[\gamma_g]$ is equivalent to the assertion that the property $U_g\vDash\varphi[\gamma_g]$ does not hold. It implies $F\setminus G_\varphi=G_\psi$.
	
	Let $G_\varphi\in\cal D$. Since $\cD$ is an ultrafilter, we have $G_\psi=F\setminus G_\varphi\notin\cal D$. So the property $U\vDash\psi[\gamma]$ does not hold. By the definition of the satisfiability, it means that $U\vDash\varphi[\gamma]$.
	
	Conversely, let $U\vDash\varphi[\gamma]$. Then the property $U\vDash\psi[\gamma]$ does not hold. Therefore $G_\psi\notin\cD$. Since $\cD$ is an ultrafilter, we have $G_\varphi=F\setminus G_\psi\in\cD$.
\end{proof}

\begin{theo}[the generalized infrafiltration theorem]\label{theo-1-infra-filt}
	Every formula $\varphi$ of the language $L(\Sigma^g_2)$ of the second-order signature $\Sigma^g_2$ with generalized equalities and belongings is infrafiltrated with respect to any ultrafilter~$\cD$ on the set~$F$.
\end{theo}

\begin{proof}
	The set of all formulas $\varphi$ of the language $L(\Sigma^g_2)$, constructed by induction from atomic formulas by means of the connectives $\lnot$ and $\land$ and the quantifier $\exists$, will be denoted by $\Psi$. The subset of the set $\Psi$, consisting of all formulas containing at most $n$ logical symbols $\lnot$, $\land$, and $\exists$, will be denoted by~$\Psi_n$. Obviously, $\Psi=\add{\Psi_n}{n\in\omega_0}$.
	
	Using the complete induction principle we shall prove the following assertion $A(n)$: \emph{every formula $\varphi\in\Psi_n$ is infrafiltrated.}
	
	If $n=0$, then $\varphi$ is an atomic formula.  By Lemma~\mref{lem-5}{infra-filt}, it is infrafiltrated. Consequently, $A(0)$ holds.
	
	Assume that for every $m<n$ the assertion $A(m)$ holds.
	Let $\varphi\in\Psi_n$. If $\varphi=\lnot\psi$, then $\psi\in\Psi_{n-1}$. Therefore, $\psi$ is infrafiltrated. By Lemma~\ref{lem-7-infra-filt}, the formula $\varphi$ is infrafiltrated as well. If $\varphi=\psi\land\xi$, then $\psi,\xi\in\Psi_{n-1}$. Therefore, by the inductive assumption, the formulas $\psi$ and $\xi$ are infrafiltrated. By Lemma~\ref{lem-6-infra-filt}, the formula $\varphi$ is infrafiltrated as well.
	Finally, if $\varphi=\exists x^\tau\psi$, then $\psi\in\Psi_{n-1}$. Consequently, as above, the formula $\psi$ is infrafiltrated. By Proposition~\ref{prop-5-infra-filt} the formula $\varphi$ is infrafiltrated as well.	Thus, the assertion $A(n)$ holds.
	
	By the complete induction principle the assertion $A(n)$ holds for every $n\in\omega_0$. This means that any formula $\varphi\in\Psi$ is infrafiltrated.
	
	Let $\varphi$ be an arbitrary formula of the language $L(\Sigma^g_2)$. Consider for $\varphi$ the accompanying formula $\varphi^*$ defined in~\ref{sec-systSigma-satisf}. By the definition of the operation $\varphi\mapsto \varphi^*$, we have $\varphi^*\in\Psi$. By the proven above, the formula $\varphi^*$ is infrafiltrated, i.\,e., $\{g\in F\mid U_g\vDash \varphi^*[\gamma_g]\}\in{\cal D} \Leftrightarrow U\vDash\varphi^*[\gamma]$.
	Proposition~\mref{prop-1}{systSigma-satisf} implies the equivalences $U\vDash\varphi^*[\gamma]\Leftrightarrow U\vDash \varphi[\gamma]$ and  $U_g\vDash\varphi^*[\gamma_g]\Leftrightarrow U_g\vDash \varphi[\gamma_g]$.
	As a result we get the following chain of equivalences:
\begin{multline*}
\bigl\{g\in F\mid U_g\vDash\varphi[\gamma_g]\bigr\}\in\cD\Leftrightarrow\bigl\{g\in F\mid U_g\vDash \varphi^*[\gamma_g]\bigr\}\in\cD \Leftrightarrow\\
\Leftrightarrow  U\vDash\varphi^*[\gamma]\Leftrightarrow U\vDash\varphi[\gamma].
\end{multline*}
 It means that the formula $\varphi$ is infrafiltrated.
\end{proof}

This theorem has one important corollary. Let $\Phi$ be some set of formulas of the language $L(\Sigma_2^g)$ of the generalized second-order signature~$\Sigma_2^g$. Let the set $\Phi$ has a model $\zzl U_0,\gamma_0\zzr$ of the signature~$\Sigma_2^g$ with true generalized equalities and belongings. Take an arbitrary set~$F$ and an arbitrary ultrafilter $\cD$ on~$F$. Consider the collection of the models $\col{\zzl U_f,\gamma_f\zzr}{f\in F}$ such that $\zzl U_f,\gamma_f\zzr\equiv\zzl U_0,\gamma_0\zzr$. The infra-$\cD$-product $\infraDprod\col{U_f}{f\in F}$ of the collection $\col{U_f}{f\in F}$ will be called 
the \emph{infra-$\cD$-power of the system $U_0$ with the exponent~$F$} and will be denoted by $\infraDpower(U_0,F)$.
The crossing $\bowtie\col{\gamma_f}{f\in F}$ of the collection $\col{\gamma_f}{f\in F}$
will be called the \emph{crossing of the evaluation $\gamma_0$ in the quantity $F$} and will be denoted by $\bowtie\zzl\gamma_0,F\zzr$.

\begin{cor}\label{cor-1-theo-1-infra-filt}
	Let $\Phi$ be some set of formulas of the language $L(\Sigma_2^g)$. If the set~$\Phi$ has a model $\zzl U_0,\gamma_0\zzr$ of the signature~$\Sigma_2^g$ with true generalized equalities and belongings, then for every set~$F$ and every ultrafilter~$\cD$ on~$F$ the set~$\Phi$ has also the model $\zzl\infraDpower\zzl U_0,F\zzr,\bowtie\zzl\gamma_0,F\zzr\zzr$ of the signature $\Sigma_2^g$ with true generalized equalities and belongings.
\end{cor}

\subsection{Compactness theorem for formulas of the language $L(\Sigma^g_2)$}\label{sec-infra-comp}

In the capacity of some pleasant complementary corollary to the infrafiltration theorem we deduce the generalized compactness theorem for the language~$L(\Sigma_2^g)$. It is well-known that it does not hold for the standard language~$L(\Sigma_2^{st})$ \cite[Appendix]{Mendelson1997}.

\begin{theo}\label{theo-2-infra-comp}
	Let $\Phi$ and $\Psi$ be some sets of formulas of the language $L(\Sigma^g_2)$ of the generalized second-order signature~$\Sigma_2^g$. Let for every finite subset~$f$ of the set~$\Phi$ the set of formulas $f+$\emph{(E1--E4)}$+\Psi$ has a model $\zzl U_f,\gamma_f\zzr$ of the signature~$\Sigma^g_2$. Then the set of formulas $\Phi+$\emph{(E1--E4)}$+\Psi$ has a model $\zzl U,\gamma\zzr$ of the signature~$\Sigma^g_2$.
\end{theo}

\begin{proof}
	Consider the set $F\equiv\{f\subset\Phi\mid 0<|f|<\omega\}$ of all finite non-empty subsets from~$\Phi$.
	
	For an element $f\in F$ consider the set $F_f\equiv\{g\in F\mid f\subset g\}$. Since $f\in F_f$, we have $F_f\ne\varnothing$. The ensemble $\mathfrak{C}\equiv\{F_f\mid f\in F\}$ has the finite intersection property, i.\,e., it is multiplicative. Hence, there is some ultrafilter $\cD$ on the set $F$ including the set ${\mathfrak C}$.
	
	Consider the system $U\equiv\infraDprod\col{U_f}{f\in F}$ and the evaluation $\gamma\equiv\;\bowtie\!\col{\gamma_f}{f\in F}$ on the system~$U$ constructed in~\ref{sec-infra-prod}. By Lemma~\mref{lem-2}{infra-prod}, $U$ is a system with the true generalized equalities and belongings.
	
	Prove that the evaluated system $\zzl U,\gamma\zzr$ is a model for the set~$\Phi$.
	
	Suppose $\varphi\in\Phi$. Consider the set $F_{\{\varphi\}}$. By condition, $U_{\{\varphi\}}\vDash\varphi[\gamma_{\{\varphi\}}]$.
	Consider the set $G_\varphi\equiv\{g\in F\mid U_g\vDash\varphi[\gamma_g]\}$. 
	If $g\in F_{\{\varphi\}}$, then $\{\varphi\}\subset g$ implies $\varphi\in g$. Therefore $U_g\vDash\varphi[\gamma_g]$. Consequently, $F_{\{\varphi\}}\subset G_\varphi$. Since $F_{\{\varphi\}}\in{\cal D}$, we have $G_\varphi\in{\cal D}$.
	
	By Theorem~\mref{theo-1}{infra-filt} we infer the property $U\vDash\varphi[\gamma]$. Thus, $\zzl U,\gamma\zzr$ is a model for the set $\Phi$. The fact that $\zzl U,\gamma\zzr$ is a model for the set $\Psi$ follows immediately from Theorem~\mref{theo-1}{infra-filt}.
\end{proof}

\section{Inductive sequence of models of non-canonical generalized second-order Dedekind real axes with exponentially increasing powers}

\subsection{The formulation of Final theorem}\label{sec-main-formul}

\begin{theofinal}\label{the-final-throrem}\mbox{}
\begin{enumerate}
\item[\textup{(I)}]
	Let $F$ be a fixed set. Then there exist some sequence $\col{\hR_i}{i\in\omega_0}$ of sets $\hR_i$, some sequence $\col{\hS_i}{i\in\omega_0}$ of superstructures~$\hS_i$ of the signature~$\Sigma^g_{R2}$ over the sets~$\hR_i$, and some sequence $\col{u_i}{i\in\omega_0}$ of mappings $u_i:\hR_i\to\hR_{i+1}$ such that:
	\begin{enumerate}
		\item[$(1)$] 
			$\hRe_0\eq\zzl\hR_0,\hS_0\zzr\eq\zzl\Rbb,S_{R2}\zzr$;
		\item[$(2)$] 
			every system $\hRe_i\eq\zzl\hR_i,\hS_i\zzr$ of the signature~$\Sigma^g_{R2}$ is a model for the theory~$Th_{R2}^g$;
		\item[$(3)$] 
			every mapping $u_i$ is an $(\approx_{\pi,i},\approx_{\pi,i+1})$-injective homomorphism of the signature~$\Sigma^g_{R2}$ from the system~$\hRe_i$ into the system~$\hRe_{i+1}$;
		\item[$(4)$]
			the image of the system~$\hRe_i$ in the system~$\hRe_{i+1}$ respectively to the homomorphism~$u_i$ is a submodel of the model~$\hRe_{i+1}$;
        \item[$(5)$]
			the support $\hR_{i+1}$ of the system $\hRe_{i+1}$ is the set $\hR_i^F$;
		\item[$(6)$]
			$(u_i p)(f)=p$ for every $p\in\hR_i$ and every $f\in F$, i.\,e., $u_i p$ is the $\{p\}$-valued function on~$F$.
	\end{enumerate}
\item[\textup{(II)}]
	There exists some superstructure $\hS_{\om_0}$ of the signature $\Sigma_{R2}^g$ over the set $\hR_{\om_0}\eq\prod\col{\hR_i}{i\in\om_0}$ and some sequence of mappings $w_i:\hR_i\to\hR_{\om_0}$ such that:
	\begin{enumerate}
		\item[$(1)$]
			the system $\hRe_{\om_0}\eq\zzl\hR_{\om_0},\hS_{\om_0}\zzr$ of the signature $\Sigma_{R2}^g$ is a model for the theory~$Th_{R2}^g$;
		\item[$(2)$] 
			every mapping $w_i$ is an $(\approx_{\pi,i},\approx_{\pi,\om_0})$-injective homomorphism of the signature~$\Sigma^g_{R2}$ from the system~$\hRe_i$ into the system~$\hRe_{\om_0}$;
		\item[$(3)$]
			the image of the system~$\hRe_i$ in the system~$\hRe_{\om_0}$ respectively to the homomorphism~$w_i$ is a submodel of the model~$\hRe_{\om_0}$;
		\item[$(4)$]
			$w_i=w_{i+1}\circ u_i$ for every $i\in\om_0$.
	\end{enumerate}
\end{enumerate}
\end{theofinal}

\subsection{Detailed superstructures in Final theorem}\mbox{}

Here we give the detailed description of the superstructures $S_i$ from Final theorem in the same manner as it is given for the superstructure $S_{R2}$ in~\ref{sec-Ded-axes}.

The superstructure $S_i$ is the quadruple $\zzl S_{c,i}, S_{e,i}, S_{b,i}, S_{v,i}\zzr$, where:
\begin{itemize}
\item 
the \emph{collection of constant structures} $S_{c,i}$ is the suit
$$
 \big((s_0^{\pi,i},s_1^{\pi,i}),\vrn,(s_0^{\rho,i},s_1^{\rho,i},s_2^{\rho,i}),(s_0^{\lm,i},s_1^{\lm,i})\big)=
 \big(0_{i},1_{i}),\vrn,(-_i,/_i,\le_i),(+_i,\cdot_i)\big);
$$
\item 
the \emph{collection of the equality ratios} $S_{e,i}$ is the suit
$
 \zzl\approx_{\pi,i},\approx_{\kap,i},\approx_{\rho,i},\approx_{\lm,i}\zzr;
$
\item 
the \emph{collection of the belonging ratios} $S_{b,i}$ is the suit
$
 \zzl\inn_{\kap,i},\inn_{\rho,i},\inn_{\lm,i}\zzr;
$
\item 
the \emph{collection of the terminals $S_{v,i}$ over the set $\hR_i$} is the suit
$$
 \zzl\pi(\hR_i),\kap(\hR_i),\rho(\hR_i),\lm(\hR_i)\zzr=\zzl\hR_i,\cP(\hR_i),\cP(\hR_i^2),\cP(\hR_i^3)\zzr.
$$
\end{itemize}

\subsection{The proof of Final theorem}\mbox{}

(I)
The construction of the infra-$\cD$-power of the system $U_0$ with the exponent~$F$ from~\ref{sec-infra-filt} gives the opportunity to prove part~I of the Final theorem.

Fix some ultrafilter $\cD$ on $F$. We shall construct the necessary sequence of models by natural induction. Take for the initial model $\hRe_0\eq\zzl\hR_0,\hS_0\zzr$ the canonical model $R_2^g\eq\zzl\Rbb,S_{R2}\zzr$ from~\ref{sec-Ded-axes}.
Assume that the model $\hRe_i\eq\zzl\hR_i,\hS_i\zzr$ with some evaluation~$\zeta_i$ is constructed.

Take the system $\hRe_{i+1}\eq\zzl\hR_{i+1},\hS_{i+1}\zzr\eq\infraDpower(\hRe_{i},F)$ and the evaluation~$\zeta_{i+1}\eq\;\bowtie\!(\zeta_i,F)$ defined in~\ref{sec-infra-prod}. According to Corollary to Theorem~\mref{theo-1}{infra-filt} the evaluated system $\zzl\hRe_{i+1},\zeta_{i+1}\zzr$ is a model for the theory~$Th^g_{R2}$. And the support~$\hR_{i+1}$ of this model is the set $\hR_i^F\eq\prod\col{\hR_{if}}{f\in F}$, where $\hR_{if}\eq\hR_i$ for every $f\in F$. 
Since the set~$\Psi_2^g$ of axioms of the theory~$Th_{R2}^g$ from~\ref{sec-Ded-axiom} consists of closed formulas only, the system~$\hRe_{i+1}$ is a model for this theory.

Define the mapping $u_i:\hR_i\to\hR_{i+1}$ setting $(u_i(p))(f)\eq p$ for every $p\in\hR_i$ and every $f\in F$. Check that $u_i$ is $(\approx_{\pi,i},\approx_{\pi,i+1})$-injective.
Take some $p,q\in\hR_i$ and suppose that $u_i(p)\approx_{\pi,i+1} u_i(q)$. By the construction from~\ref{sec-infra-prod} there exists $G\in\cD$ such that $(u_i(p))(g)\approx_{\pi,g,i}(u_i(q))(g)$ for every $g\in G$. Since $G\ne\vrn$ we can take $g_0\in G$. Then $(u_i(p))(g_0)\eq p$ and $(u_i(q))(g_0)\eq q$ implies $p\approx_{\pi,i} q$.

The construction of constant structures presented in~\ref{sec-infra-prod} implies immediately that~$u_i$ is a homomorphism of the signature~$\Sigma^g_{R2}$ from the system~$\hRe_i$ into the system~$\hRe_{i+1}$.

\medskip

(II)
The construction of the infra-{$\cD$}-product of the collection of mathematical systems from~\ref{sec-infra-prod} gives the opportunity to prove part~II of the Final theorem. 
Fix some ultrafilter~$\cE$ on~$\om_0$. Take the system $\hRe_{\om_0}\eq\zzl\hR_{\om_0},\hS_{\om_0}\zzr\eq\infraEprod\col{\hRe_i}{i\in\om_0}$ and the evaluation $\zeta_{\om_0}\eq\;\bowtie\!\col{\zeta_i}{i\in\om_0}$ defined in~\ref{sec-infra-prod}.
According to part~I and Theorem~\mref{theo-1}{infra-filt} the evaluated system $\zzl\hRe_{\om_0},\zeta_{\om_0}\zzr$ is a model for the theory~$Th_{R2}$. Since the set~$\Psi_2^g$ of axioms of this theory from~\ref{sec-Ded-axiom} consists of closed formulas only, the system~$\hRe_{\om_0}$ is a model for this theory.

Fix $i\in\om_0$. Construct some mapping $w_i:\hR_i\to\hR_{\om_0}$ by the inverse and direct natural inductions. 
For the base of direct induction put $(w_i p)(i)\eq p$ and $(w_i p)(i+1)\eq u_i p$. 
For the step of direct induction put $(w_i p)(j+1)\eq u_j((w_i p)(j))$ for $j\ge i+1$.
Fix some $f_0\in F$. Put $(w_i p)(i-1)\eq p(f_0)$ for the base of inverse induction. 
For the step of inverse induction put $(w_i p)(j-1)\eq ((w_i p)(j))(f_0)$ for $1\le j\le i-1$.
These constructions can be described in a more rigorous form based on \cite[Theorem 1(1.2.8)]{ZakhRodi2018SFM1}.

By the natural induction in ST it can be proved that $w_i$ is a homomorphism of the signature~$\Sigma^g_{R2}$ from the system~$\hRe_i$ into the system~$\hRe_{\om_0}$ (see the example of scrupulous arguments below).

Check that $w_i$ is $(\approx_{\pi,i},\approx_{\pi,\om_0})$-injective. Take some $p,q\in\hR_i$ and suppose that $w_i(p)\approx_{\pi,\om_0} w_i(q)$. 
By the construction from~\ref{sec-infra-prod} there exists $J\in\cE$ such that $(w_i p)(j)\approx_{\pi,j} (w_i q)(j)$ for every $j\in J$. Consider the binary partition $\zzl i,\om_0\razn i\zzr$ of~$\om_0$. Since~$\cE$ is a ultrafilter, we infer that either $i\in\cE$ or $\om_0\razn i\in\cE$. If $i\in\cE$, then $\om_0\in\cE$ but it is not so. Hence, $\om_0\razn i\in\cE$. This implies $J\cap(\om_0\razn i)\in\cE$, and, therefore, there is $j\in J$ such that $j\ge i$. Take $k_0\eq j-i$.

If $j=i$, then by the definition of $w_i$ we have $(w_i p)(i)=p$ and $(w_i q)(i)=q$. Hence, $p\approx_{\pi,i} q$. If $j=i+1$, then $(w_i p)(i+1)\eq u_i p$ and $(w_i q)(i+1)\eq u_i q$ imply $(u_i p)(i+1)\approx_{\pi,i+1}(u_i q)(i+1)$. Since by assertion~3 of part~I the mapping~$u_i$ is $(\approx_{\pi,i},\approx_{\pi,i+1})$-injective, we infer that $p\approx_{\pi,i} q$.
Consider in ST the set $K\eq\{k\in\Nbb\mid ((w_i p)(i+k)\approx_{\pi,i+k}(w_i q)(i+k))\impl p\approx_{\pi,i} q\}$. Let~$\Phi_{ST}$ be a totality of axioms of the theory ST, i.\,e., $\Phi_{ST}$ consists of all explicit proper axioms of this theory, all implicit proper axioms of this theory, and all implicit logical axioms of the predicate calculus (see, e.\,g., \cite[1.1.3--1.1.11 and A.1.2]{ZakhRodi2018SFM1}). Denote the first formula in the definition of~$K$ by~$\vphi(i+k)$ and the second one by~$\psi$. We have proved in ST the existence of deduction $\Phi_{ST},\vphi(i+1)\vdash\psi$. Since ST is the first-order theory, we conclude that $\Phi_{ST}\vdash\vphi(i+1)\impl\psi$ by virtue of the deduction theorem (see, for example, \cite[Proposition 2.5]{Mendelson1997} or \cite[1.1.3]{ZakhRodi2018SFM1}). Hence, $1\in K$.

Suppose that $k\in K$ and $(w_i p)(i+k+1)\approx_{\pi,i+k+1}(w_i q)(i+k+1)$.
By the definition of~$w_i$ we have $(w_i p)(i+k+1)\eq u_{i+k}((w_i p)(i+k))$ and the same for~$q$. Since by assertion~3 of part~I the mapping~$u_{i+k}$ is $(\approx_{\pi,i+k},\approx_{\pi,i+k+1})$-injective, we infer that 
$(w_i p)(i+k)\approx_{\pi,i+k}(w_i q)(i+k)$. Now from $k\in K$ we deduce that $p\approx_{\pi,i}q$. Thus, we have proved the existence of deduction $\Phi_{ST},\vphi(i+k+1)\vdash\psi$. As above this implies $\Phi_{ST}\vdash\vphi(i+k+1)\impl\psi$, and, therefore, $k+1\in K$. By the principle of natural induction in ST (see \cite[1.2.6]{ZakhRodi2018SFM1}) we get $K=\Nbb$.

This means that for our $j=i+k_0$ we have $k_0\in K$. Since $j\in J$, we conclude that $p\approx_{\pi,i}q$. This proves assertion~2.

Now we must only prove assertion 4. Fix $p\in\hR_i$. Then $(u_i p)(f)\approx p$ for every $f\in F$. By the definition we have $(w_{i+1}(u_i p))(i+1)\approx u_i p\approx (w_i p)(i+1)$. For the base of direct induction we have 
$$
 (w_{i+1}(u_i p))(i+2)\approx u_{i+1}((w_{i+1}(u_i p))(i+1))\approx u_{i+1}((w_i p)(i+1))\approx(w_i p)(i+2).
$$
For the base of inverse induction we have 
$$
 (w_{i+1}(u_i p))(i)\approx (u_i p)(f_0)\approx p\approx (w_i p)(i).
$$
Then by the direct and inverse inductions we check that $(w_{i+1}(u_i p))(j)\approx_{\pi,j}(w_i p)(j)$ for every $j\in\om_0$. Hence, $(w_{i+1}\circ u_i)(p)\approx_{\pi,\om_0}w_i(p)$ for every $p\in\hR_i$.
\qed
\bigskip
	
\begin{remark}\label{A3-A19-remark}
	Since every set $\hR_i\eq\hR_{i-1}^F$ for $i\ge 1$ consists of ``real''-valued functions $p:F\to\hR_{i-1}$, it is necessary to clarify directly the satisfaction of non-evident axioms~A3 (the existence and functionality of the inversion) and~A19 (the linearity of the order) on the systems~$\hRe_i$.
	
	In case of A3 take any function $p\in\hR_i$ such that $p\not\approx_{\pi,i}0_i$, where $0_i$ denotes the null in~$\hR_i$. 
	Consider the binary partition of~$F$ consisting of the sets $\zer(p)\eq\{f\in F\mid p(f)\approx_{\pi,f,i-1}0_{i-1}\}$ and $\coz(p)\eq F\razn\zer(p)$.
	Since~$\cD$ has the binary partition property, we have either $\zer(p)\in\cD$ or $\coz(p)\in\cD$. In the first case we conclude that $p\approx_{\pi,i} 0_i$ but it contradicts our assumption. Hence, $\coz(p)\in\cD$ and $p(g)\not\approx_{\pi,i-1}0_{i-1}$ for every $g\in\coz(p)$. By~A3 for every $g\in\coz(p)$ there exists $p(g)^{-1}$ such that $(p(g),p(g)^{-1})\inn_{\rho,i-1}/_{i-1}$. Define $p^{-1}$ setting $p^{-1}(g)\eq p(g)^{-1}$ for every $g\in\coz(p)$ and $p^{-1}(f)\eq p(f)$ for every $f\in\zer(p)$. 
	By the definition of $\rho$-belonging~$\inn_{\rho,i}$ from~\ref{sec-infra-prod} $(p,p^{-1})\inn_{\rho,i}/_{i}$. Thus, we deduced the existence of the inversion in~$\hRe_i$ from the existence of the inversion in~$\hRe_{i-1}$ using the binary partition property of the ultrafilter~$\cD$.
	
	In case of A19 take any functions $p,q\in\hR_{i}$. Since $\hR_{i-1}$ is linearly ordered with respect to the order~$\le_{i-1}$, we can take the binary partition of~$F$ consisting of the sets 
\begin{align*}
G&\eq\{g\in F\mid (p(g),q(g))\inn_{\rho,i-1}\le_{i-1}\} \text{ and }\\
H'&\eq\{h\in F\mid ((q(h),p(h))\inn_{\rho,i-1}\le_{i-1})\wedge(q(h)\not\approx_{\pi,i-1}p(h))\}.
\end{align*}
 By binary partition property of~$\cD$ we have either $G\in\cD$ or $H'\in\cD$. In the first case we conclude that $(p,q)\inn_{\rho,i}\le_{i}$.
	In the second case we can see that $H'\subset H\eq\{h\in F\mid (q(h),p(h))\inn_{\rho,i-1}\le_{i-1}\}\in\cD$ implies  $(q,p)\inn_{\rho,i}\le_{i}$.
	Thus, we deduced the linearity of the order in~$\hRe_i$ from the linearity of the order in~$\hRe_{i-1}$ using again the binary partition property of the ultrafilter~$\cD$.
\end{remark}

\begin{OpQu}
Part II of Final theorem shows that the model $\hRe_{\om_0}$ can be considered as some pretender for the \emph{inductive limit of the inductive sequence $s\eq\col{\hRe_i}{i\in\om_0}$} in the sense of \cite[11.8]{Semadeni1971}. 
But this is an open question.
\end{OpQu}

\begin{OpQu}
[\emph{about transfinite extension of the inductive sequence~$s$}]
Let $\lm$ be an ordinal number such that $\lm>\om_0$ and~$\cE$ be an ultrafilter on~$\lm$. Since~$\cE$ has the binary partition property, we can consider the ultrafilters $\cE_{\al}\eq\{E\subset\al\mid\exists\,G\in\cE\ (E=G\cap\al)\}$ for every ordinal number $\al\in\lm\razn\om_0=[\om_0,\lm[$. Starting from $\hRe_{\om_0}$ we can construct by the transfinite procedure some collection $t\eq\col{\hRe_{\al}}{\al\in\lm\razn\om_0}$ of models for the theory $Th^g_{R2}$ such that:
1) $\hRe_{\al}\eq\infraEalprod\col{\hRe_{\gm}}{\gm\in\al\razn\om_0}$ for limit ordinal number~$\al$ and
2) $\hRe_{\al+1}\eq\infraEalpower(\hRe_{\al},F)$. \emph{Is the collection~$t$ is inductive with respect to some injective homomorphisms $u_{\al\beta}:\hR_{\al}\to\hR_{\beta}$ for every $\al<\beta$ and does~$t$ extend~$s$?}
\end{OpQu}

\begin{Suppl}
In \cite[C.3.4]{ZakhRodi2018SFM1} the \emph{generalized second-order Peano\,--\,Landau theory $Th^g_{N2}$ of natural numbers} is considered. 
It is clear that some inductive sequence $\col{N_i}{i\in\om_0}$ of models of this theory can be constructed, which is similar to the inductive sequence $s\eq\col{\hRe_i}{i\in\om_0}$ constructed above. And also the inductive ``quasilimit'' $N_{\om_0}$ of this sequence can be constructed similarly to to the inductive ``quasilimit'' $\hRe_{\om_0}$.
Moreover, its own Final theorem can be proved for the generalized models $N_{i}$ and $N_{\om_0}$ of the theory $Th^g_{N2}$. Besides, open questions~1 and~2 are valid for these hypothetical models.
\end{Suppl}

\bibliographystyle{elsarticle-num} 
\bibliography{ZaRoDedekind}

\end{document}